\RequirePackage{fix-cm}
\documentclass[smallcondensed]{svjour3}     
\smartqed  
\usepackage{cite}
\usepackage{amsmath}
\usepackage{amsfonts}
\usepackage{amssymb}
\usepackage{bm}
\usepackage{subcaption}
\usepackage{caption}
\usepackage{xspace}
\usepackage{afterpage}
\usepackage{color}
\usepackage{stmaryrd}
\usepackage{textcomp}
\usepackage[textwidth=1.5in,textsize=small]{todonotes}

\usepackage{hyperref}
\hypersetup{
  backref=true,
  colorlinks=true,
  linkcolor=blue,
  citecolor=blue,
  urlcolor=blue
}

\newcommand{\fnc}[1]{\ensuremath{\mathcal{#1}}}

\newcommand{\mat}[1]{\ensuremath{\mathsf{#1}}}

\newcommand{\mr}[1]{\mathrm{#1}}
\newcommand{\diff}[0]{\mr{d}}
\newcommand{\Tr}[0]{^\mr{T}}

\renewcommand{\H}[0]{\mat{H}}
\newcommand{\Hk}[0]{\H_{\kappa}}
\newcommand{\D}[0]{\mat{D}}
\newcommand{\Dx}[0]{\mat{D}_{x}}
\newcommand{\Dy}[0]{\mat{D}_{y}}

\newcommand{\Skew}[0]{\mat{S}}
\newcommand{\Sx}[0]{\Skew_{x}}

\newcommand{\Q}[0]{\mat{Q}}
\newcommand{\Qx}[0]{\mat{Q}_{x}}

\newcommand{\Ex}[0]{\mat{E}_{x}}

\newcommand{\R}[0]{\mat{R}}
\newcommand{\B}[0]{\mat{B}}
\newcommand{\Bg}[0]{\mat{B}_{\gamma}}
\newcommand{\Gk}[0]{\mat{F}_{\kappa}}
\newcommand{\Gn}[0]{\mat{F}_{\nu}}
\newcommand{\Cgk}[0]{\mat{C}_{\gamma\kappa}}
\newcommand{\Cgn}[0]{\mat{C}_{\gamma\nu}}
\newcommand{\Lxk}[0]{\bm{l}_{x,\kappa}^{\gamma}}
\newcommand{\Lyk}[0]{\bm{l}_{y,\kappa}^{\gamma}}
\newcommand{\Lxn}[0]{\bm{l}_{x,\nu}^{\gamma}}
\newcommand{\Lyn}[0]{\bm{l}_{y,\nu}^{\gamma}}

\newcommand{\Sig}[0]{\mat{\Sigma}}
\newcommand{\Lam}[0]{\mat{\Lambda}}
\newcommand{\Lamxx}[0]{\Lam_{xx}}
\newcommand{\Lamxy}[0]{\Lam_{xy}}
\newcommand{\Lamyx}[0]{\Lam_{yx}}
\newcommand{\Lamyy}[0]{\Lam_{yy}}

\newcommand{\pk}{\bm{p}_{k}}

\newcommand{\nk}[0]{n_{\kappa}}

\newcommand{\poly}[1]{\mathbb{P}_{#1}}

\newcommand{\Nx}[0]{\mat{N}_x}
\newcommand{\Ny}[0]{\mat{N}_y}
\newcommand{\Nxg}[0]{\mat{N}_{x,\gamma}}
\newcommand{\Nyg}[0]{\mat{N}_{y,\gamma}}
\newcommand{\Rg}[0]{\mat{R}_{\gamma}}
\newcommand{\Rgk}[0]{\mat{R}_{\gamma\kappa}}
\newcommand{\Rgn}[0]{\mat{R}_{\gamma\nu}}

\newcommand{\Dgk}[0]{\mat{D}_{\gamma\kappa}}
\newcommand{\Dgn}[0]{\mat{D}_{\gamma\nu}}
\newcommand{\Ukp}[0]{\bm{u}_\kappa}
\newcommand{\Vkp}[0]{\bm{v}_\kappa}
\newcommand{\Unu}[0]{\bm{u}_\nu}
\newcommand{\Vnu}[0]{\bm{v}_\nu}
\newcommand{\Wkp}[0]{\bm{w}_\kappa}
\newcommand{\Wnu}[0]{\bm{w}_\nu}
\newcommand{\Ugam}[0]{\bm{u}_\gamma}
\newcommand{\W}[0]{\bm{w}}
\newcommand{\SumAll}{\sum_{\Omega_\kappa \in \fnc{T}_h}}

\newcommand{\Mk}[0]{\mat{M}_{\kappa}}

\newcommand{\Uh}{\bm{u}_{h}}
\newcommand{\Vh}{\bm{v}_{h}}
\newcommand{\Ug}{\bm{u}_{\gamma}}
\newcommand{\Vg}{\bm{v}_{\gamma}}
\newcommand{\Siggk}[1]{\Sig_{\gamma\kappa}^{(#1)}}
\newcommand{\Siggn}[1]{\Sig_{\gamma\nu}   ^{(#1)}}
\newcommand{\Siggam}[1]{\Sig_{\gamma}^{(#1)}}
\newcommand{\phnt}[1]{\phantom{#1}}

\newcommand{\LL}{\llbracket}
\newcommand{\RR}{\rrbracket}

\newcommand{\I}[0]{\mat{I}} 

\DeclareMathOperator{\mydiag}{diag}

\newcommand{\eg}[0]{{e.g.\@}\xspace}
\newcommand{\ie}[0]{{i.e.\@}\xspace}
\newcommand{\etc}[0]{{etc.\@}\xspace}

\newcommand{\resp}[0]{{resp.\@}\xspace}

\newcommand{\ignore}[1]{} 

\journalname{Journal of Scientific Computing}
\titlerunning{Interior Penalties for SBP discretizations}

\begin{document}

\title{Interior Penalties for Summation-by-Parts Discretizations of Linear Second-Order Differential Equations}

\author{Jianfeng Yan \and
  Jared Crean \and
  Jason E. Hicken}

\institute{Jianfeng Yan, Graduate Student\at
  Rensselaer Polytechnic Institute \\
  \email{yanj4@rpi.edu}
  \and
  Jared Crean, Graduate Student\at
  Rensselaer Polytechnic Institute \\
  \email{creanj@rpi.edu}
  \and
  Jason E.\ Hicken, Assistant Professor\at
  Rensselaer Polytechnic Institute \\
  \email{hickej2@rpi.edu}
}

\date{}

\maketitle
\begin{abstract}
  This work focuses on multidimensional summation-by-parts (SBP) discretizations
  of linear elliptic operators with variable coefficients.  We consider a
  general SBP discretization with dense simultaneous approximation terms (SATs),
  which serve as interior penalties to enforce boundary conditions and
  inter-element coupling in a weak sense.  Through the analysis of adjoint
  consistency and stability, we present several conditions on the SAT penalties.
  Based on these conditions, we generalize the modified scheme of Bassi and Rebay
  (BR2) and the symmetric interior penalty Galerkin (SIPG) method to SBP-SAT
  discretizations.  Numerical experiments are carried out on unstructured grids
  with triangular elements to verify the theoretical results.
\end{abstract}
\section{Introduction}

Summation-by-parts (SBP) operators~\cite{Kreiss1974, Strand1994} are high-order
finite-difference operators that can be used to construct time-stable
discretizations.  The stability and high-order accuracy of SBP discretizations
makes them attractive for simulating conservation laws over long periods of
time.  Diagonal-norm SBP operators are especially appealing for compressible
fluid flow simulations, because they can be used to construct discretizations
that are entropy stable without relying on exact integration~\cite{Fisher2013b,
  Crean2016}.

A drawback of classical SBP operators is that, like all one-dimensional
operators, they require multiblock tensor-product grids in order to be applied
to complex geometries; generating high-quality, multiblock, hexahedral grids is
difficult to automate and time consuming when done manually.  Motivated by this
drawback, Hicken, Del Rey Fern\'andez, and Zingg~\cite{multiSBP} generalized the
SBP definition to arbitrary bounded domains.  They showed that a degree $p$
diagonal-norm SBP operator can be constructed on a given domain provided a
degree $2p-1$ cubature exists whose nodes produce a full-rank Vandermonde
matrix.

Ultimately, we are interested in using diagonal-norm, multidimensional SBP
operators to construct high-order, entropy-stable discretizations of the
Navier-Stokes equations on unstructured grids.  We believe this approach will
combine the stability and accuracy advantages of classical SBP operators with the
flexibility of tetrahedral grids; however, there are several developments
necessary to bring multidimensional SBP methods to the level of maturity of
classical SBP methods.  The development considered in this work is the
discretization of linear elliptic operators with variable coefficients, which
are a prototype for the viscous terms in the Navier-Stokes equations.

In this paper we do \emph{not} address the construction of specialized
second-derivative SBP operators~\cite{Mattsson2004b, Mattsson2012,
  Fernandez2015}.  Instead, we adopt the so-called ``first-derivative twice''
approach.  For classical finite-difference methods, applying the first-\linebreak
derivative twice approximately doubles the stencil size and is typically less
accurate~\cite{Mattsson2012}; however, the multidimensional SBP operators that
we intend to use are dense on each element, so the stencil size is not changed
by applying the first-derivative operator twice.  That said, we believe that
more general multidimensional SBP operators, particularly those that are sparse
on each element, would benefit from a generalization of the works
in~\cite{Mattsson2012} and~\cite{Fernandez2015}.

\ignore{
The purpose of this paper is \emph{not} to construct second-derivative SBP
operators.  For classical one-dimensional SBP operators, narrow
stencil\footnote{Here, narrow stencil refers to explicit finite-difference
  operators (\ie not compact) of a given accuracy that use a minimal number of
  points in their stencil; see~\cite{Mattsson2004b}} second-derivative operators
were studied by Mattsson and Nordstr\"{o}m~\cite{Mattsson2004b} in the constant
coefficient case, and by Mattsson~\cite{Mattsson2012} in the variable
coefficient case. Fern\'andez and Zingg~\cite{Fernandez2015} analyzed
second-derivative operators for generalized SBP operators.

The alternative to using specialized second-derivative operators is to apply
first-derivative SBP operators twice, and this is the approach we adopt.  For
classical finite-difference methods, applying the first-derivative twice
approximately doubles the stencil size and is typically less
accurate~\cite{Mattsson2012}; however, the multidimensional SBP operators that
we intend to use are dense on each element, so the stencil size is not changed
by applying the first-derivative operator twice.  Nevertheless, we believe that
more general multidimensional SBP operators, particularly those that are sparse
on each element, would greatly benefit from a generalization of the works
in~\cite{Mattsson2012} and~\cite{Fernandez2015}.
}

Rather than SBP operators themselves, the focus of this paper is the analysis of
penalty terms to enforce boundary conditions and inter-element coupling for
multidimensional SBP discretizations of elliptic and parabolic partial
differential equations (PDEs).  In the SBP literature, such penalty terms are
called simultaneous approximation terms (SATs) and were popularized by~\cite{Carpenter1999}.  SATs for multidimensional SBP discretizations
of the linear advection equation were recently studied in~\cite{Fernandez2016}, and SATs
for tensor-product SBP discretizations of second-order PDEs have been
investigated by a number of authors; see, \eg, \cite{Gong2011} and the
review~\cite{Carpenter2010}.

SATs are analogous to interior penalties (IPs) used in the finite-element (FE)
community; see the review~\cite{Arnold2002} and the references therein.  Indeed,
there are strong connections between FE and SBP discretizations, and the present
work draws heavily from the discontinuous Galerkin (DG) literature.

Despite the extensive prior work in this area, the present work makes the
following contributions.
\begin{enumerate}
\item We present the requirements on \emph{dense} SAT coefficient matrices to
  obtain multidimensional SBP-SAT discretizations that are simultaneously
  consistent, conservative, adjoint consistent, and stable.
\item We generalize the modified scheme of Bassi and Rebay
  (BR2)~\cite{Bassi2005} and the symmetric interior penalty Galerkin (SIPG)
  method~\cite{douglas:1976, Arnold2002, Hartmann2005} to multidimensional SBP
  discretizations. In addition, we show how, in the SBP-SAT context, SIPG can be
  derived from BR2 using matrix analysis.
\item We show that a particular SBP-SAT implementation of the BR2 penalty has a
  computational cost of $\text{O}(p)$ in 2D and $\text{O}(p^2)$ in 3D, in
  contrast with $\text{O}(p^2)$ and $\text{O}(p^3)$, respectively, for DG
  implementations with Lagrange basis functions.
\end{enumerate}
The above results do not rely on exact integration, so they are valid for fully
discretized elliptic problems and semi-discretized parabolic problems.

The remaining sections are organized as follows.  After introducing notation and
reviewing SBP operators, Section~\ref{sec:discrete} presents the model parabolic
PDE and its SBP-SAT discretization.  Section~\ref{sec:adjoint} investigates the
adjoint consistency of the discretization and delineates the necessary
adjoint-consistency conditions on the SAT penalties.  The penalties are further
constrained by the energy-stability analysis in
Section~\ref{sec:energy_analysis}.  The resulting conditions are used to
generalize the BR2 and SIPG methods to multidimensional SBP discretizations in
Section~\ref{sec:generalize}.  Verification studies are provided in
Section~\ref{sec:results}, and a summary is provided in
Section~\ref{sec:conclude}.

\section{Multi-dimensional SBP discretization of parabolic PDEs}\label{sec:discrete}

\subsection{Notation}

Functions are denoted with capital letters in calligraphic font; for example
$\fnc{U} \in L^{2}(\Omega)$ is a square-integrable function on the domain
$\Omega$.  A function evaluated on a node set is denoted by a lowercase letter
in bold font.  For example, the function $\fnc{U}$ evaluated at the nodes of
$S=\left\{(x_{i},y_{i})\right\}_{i=1}^{n}$ is given by
\begin{equation*}
  \bm{u} = \begin{bmatrix}
  \fnc{U}(x_1,y_1) & \fnc{U}(x_2,y_2) & \cdots & \fnc{U}(x_n,y_n)
  \end{bmatrix}^T.
\end{equation*}
The space of polynomials of total degree $p$ in $x$ and $y$ on $\Omega$ is
denoted by $\poly{p}(\Omega)$.  As with generic functions, a polynomial that is
evaluated at the points of $S$ will be represented using its corresponding
lowercase letter in bold font; for example, for $\fnc{P} \in \poly{p}(\Omega)$
we would have
\begin{equation*}
  \bm{p} \equiv \begin{bmatrix} \fnc{P}(x_{1},y_{1}) & \fnc{P}(x_2,y_2) & \cdots & 
    \fnc{P}(x_{n},y_{n}) \end{bmatrix}^T.
\end{equation*}

Matrices are represented with an uppercase sans-serif type, for example $\mat{A}
\in \mathbb{R}^{n\times m}$. 

\subsection{SBP definition and face operators}

We adopt the definition of multidimensional SBP operators proposed in
\cite{multiSBP}.  To keep the presentation self-contained, the definition for an
operator approximating $\partial/\partial x$ on a two dimensional domain is
provided below.  The definition for the SBP operator approximating
$\partial/\partial y$ is analogous.

\begin{definition}\label{def:SBP}
  {\bf Two-dimensional summation-by-parts operator:} Consider an open and
  bounded domain $\Omega_\kappa \subset \mathbb{R}^{2}$ with a piecewise-smooth
  boundary $\partial \Omega_\kappa$.  The matrix $\Dx$ is a degree $p$ SBP
  approximation to the first derivative $\frac{\partial}{\partial x}$ on the
  nodes $S_{\kappa}=\left\{(x_{i},y_{i})\right\}_{i=1}^{\nk}$ if
  \begin{enumerate}
  \item For all $\fnc{P} \in \poly{p}(\Omega_{\kappa})$, the vector $\Dx\pk$ is
    equal to $\partial \fnc{P}/\partial x$ at the nodes $S_{\kappa}$;
    \label{sbp:accuracy}
  \item $\Dx = \H^{-1}\Qx$, where $\H$ is symmetric positive-definite,
    and; \label{sbp:H}
  \item $\Qx = \Sx + \frac{1}{2}\Ex$, where $\Sx^T=-\Sx$, $\Ex^T=\Ex$, and
    $\Ex$ satisfies
    \begin{equation*}
      \bm{p}^T\Ex\bm{q} =\displaystyle\oint_{\Gamma} \fnc{P} \fnc{Q} n_{x}
      \mr{d}\Gamma,
      \qquad\forall\; \fnc{P}, \fnc{Q} \in \poly{r}(\Omega_{\kappa}),
    \end{equation*}
    where $r \ge p$, and $n_{x}$ is the $x$ component of
    $\bm{n}=\left[n_{x},n_{y}\right]\Tr$, the outward pointing unit normal on
    $\partial \Omega_\kappa$. \label{sbp:Ex}
  \end{enumerate}  
\end{definition}

The subsequent analysis is restricted to so-called diagonal-norm SBP operators,
that is, SBP operators for which $\H$ is a diagonal matrix with positive
entries.  In this case, it was shown in~\cite{multiSBP} that the nodes
$S_{\kappa}$ and diagonal entries of $\H$ define a cubature rule that is exact
for polynomials of total degree $2p-1$.

In order to define SATs for multidimensional SBP operators, we follow
References \cite{Hicken2016, Fernandez2016} and introduce interpolation/extrapolation
operators from the SBP element nodes to cubature nodes on the faces of the
elements.  For example, consider an element $\Omega_\kappa$ with a piecewise
smooth boundary $\partial \Omega_\kappa$, and let $\gamma \subset \partial
\Omega_\kappa$ denote one of its faces.  Let $S_{\gamma} =
\{(x_j,y_j)\}_{j=1}^{n_{\gamma}} \subset \gamma$ be a set of cubature nodes with
corresponding positive weights $\{ b_{j} \}_{j=1}^{n_{\gamma}}$ that is exact
for polynomials of degree $2r$, where $r \ge p$.  The matrix
$\Rgk \in \mathbb{R}^{n_{\gamma}\times n_{\kappa}}$ is a
degree $r$ interpolation/extrapolation operator from the SBP nodes $S_{\kappa}$
to the face nodes $S_{\gamma}$ if, for all $\fnc{P} \in \poly{r}(\Omega_{\kappa})$,
\begin{equation*}
  \left(\Rgk \pk\right)_{j} = 
  \sum_{i=1}^{\nk} (\Rgk)_{ji} \fnc{P}(x_{i},y_{i}) 
  = \fnc{P}(x_{j},y_{j}),
  \qquad\forall j=1,2,\ldots,n_{\gamma}.
\end{equation*}

For a given (strong) cubature rule of degree $2p-1$ defined on $\Omega_\kappa$, it was
shown in \cite{Fernandez2016} that there exists at least one SBP operator whose
corresponding matrix $\Ex$ has the decomposition
\begin{equation}
  \Ex = \sum_{\gamma \subset \partial \Omega_\kappa} \Rgk^{T} \Nxg \Bg \Rgk,
  \label{eq:SATdecomp}
\end{equation}
where $\B_{\gamma} = \mydiag\left(b_{1},b_{2},\ldots,b_{n_{\gamma}}\right)$ is
an $n_{\gamma}\times n_{\gamma}$ diagonal matrix holding the cubature weights for
$\gamma$ along its diagonal, and $\Nxg =
\mydiag\left(n_{x,1},n_{x,2},\ldots,n_{x,n_{\gamma}}\right)$ is an
$n_{\gamma}\times n_{\gamma}$ diagonal matrix holding the $x$ component of the
outward unit normal with respect to $\Omega_\kappa$ at the cubature points of $\gamma$.
We will assume in the following analysis that the SBP operators are such that
$\Ex$ has the decomposition~\eqref{eq:SATdecomp}, and that the operators in the
$y$ direction have analogous decompositions.

\ignore{
SBP discretizations are distinguished characteristics of SBP discretizations are
that
\begin{enumerate}
\item they can be expressed equivalently in (pointwise) strong or weak form, and
\item their solution is defined pointwise and not, in general, in terms of a
  (unique) basis.
\end{enumerate}
Some FE methods exhibit the first characteristic; for example, spectral collocation
methods based on Legendre-Gauss or Legendre-Gauss-Lobatto quadrature
points~\cite{Gassner}.  To the best of our knowledge, the second characteristic
is not shared with FE methods.
}

\subsection{The model PDE}

We consider the following linear parabolic PDE --- or the
corresponding steady Poisson PDE --- defined on the compact domain $\Omega
\subset \mathbb{R}^{2}$:
\begin{equation}\label{eq:parabolic}
    \frac{\partial \fnc{U}}{\partial t} - \nabla\cdot\left( \Lambda \nabla \fnc{U} \right) = \fnc{F},\quad \forall \; (x,y) \in \Omega,
\end{equation}
where $\fnc{F} \in L^{2}(\Omega\times[0,T])$ is a given source term and
\begin{equation*}
\Lambda \equiv \begin{bmatrix} \lambda_{xx} & \lambda_{xy}\\ 
  \lambda_{yx} & \lambda_{yy} \end{bmatrix}
\end{equation*}
is a symmetric, positive-definite tensor.  The parabolic PDE is provided with
the initial condition
\begin{equation}\label{eq:IC}
  \fnc{U}(0,x,y) = \fnc{U}_{0}(x,y), \quad \forall \; (x,y) \in \Omega,
\end{equation}
where $\fnc{U}_{0} \in L^2(\Omega)$.  Finally, the PDE is supplied with the
Dirichlet and Neumann boundary conditions,
\begin{equation}\label{eq:BCs}
  \begin{aligned}
    &\fnc{U}(t,x,y) = \fnc{U}_\fnc{D}(t,x,y), \quad \forall \; (x,y) \in \Gamma^\fnc{D}, \\[2ex]
    &\hat{\bm{n}} \cdot \left( \Lambda\nabla \fnc{U}(t,x,y)\right) = \fnc{U}_\fnc{N}(t,x,y),
    \quad \forall \; (x,y) \in \Gamma^\fnc{N},
  \end{aligned}
\end{equation}
where $\fnc{U}_\fnc{D} \in L^2(\Gamma^\fnc{D}\times[0,T])$ and $\fnc{U}_\fnc{N}
\in L^2(\Gamma^\fnc{N}\times[0,T])$.  The vector $\hat{\bm{n}} = [n_x,n_y]^T$ is
the outward pointing unit normal on the boundary $\partial\Omega$.  We assume that the
Dirichlet boundary is nonempty, $\Gamma^\fnc{D} \ne \emptyset$, so that the
solution is unique.  Furthermore, $\Gamma = \Gamma^\fnc{D} \cup \Gamma^\fnc{N}$
and $\Gamma \setminus \Gamma^\fnc{D} = \Gamma^\fnc{N}$.

\subsection{Strong-form Discretization}

Let $\fnc{T}_h = \bigcup_{\kappa=1}^{K} \Omega_\kappa$ denote a partition of
the domain $\Omega$ into $K$ SBP elements, where $\Omega_\kappa$ denotes the
domain of the $\kappa$th element.  The discrete solution on element $\Omega_\kappa$
will be represented by the vector $\bm{u}_{\kappa} \in \mathbb{R}^{\nk}$ whose
entries are the discrete solution at the SBP nodes $S_{\kappa}$.  The global
discrete solution, denoted $\bm{u}_h \in \mathbb{R}^{\sum \nk}$, is the
concatenation of all elementwise solutions.

A consistent SBP-SAT semi-discretization of \eqref{eq:parabolic} on element $\kappa$
is given by
\begin{equation}\label{eq:parabolic_SBP}
\frac{\diff \bm{u}_{\kappa}}{\diff t} = \D_\kappa\bm{u}_{\kappa} + \bm{f}_\kappa 
-\Hk^{-1}\bm{s}_{\kappa}^{\fnc{I}}\left(\bm{u}_{h}\right) 
-\Hk^{-1}\bm{s}_{\kappa}^{\fnc{B}}\left(\bm{u}_{h}, \bm{u}_\fnc{D}, \bm{u}_\fnc{N}\right),
\end{equation}
where $\bm{f}_{\kappa}$ is $\fnc{F}$ evaluated at the nodes of element $\Omega_\kappa$, 
and
\begin{equation}\label{eq:Dkappa}
\D_\kappa = \left\{ \begin{bmatrix} \Dx & \Dy \end{bmatrix}
\begin{bmatrix} \Lamxx & \Lamxy \\ \Lamyx & \Lamyy \end{bmatrix}
\begin{bmatrix} \Dx \\ \Dy \end{bmatrix} \right\}_{\kappa}
\end{equation}
is the SBP approximation of $\nabla\cdot(\Lambda \nabla)$ on element
$\Omega_\kappa$, with $\Dx \in \mathbb{R}^{\nk\times \nk}$ and $\Dy \in
\mathbb{R}^{\nk \times \nk}$ the first-derivative SBP operators in the $x$ and
$y$ directions, respectively.  The subscript notation, $()_\kappa$, indicates a
vector or operator on element $\kappa$.  The Cartesian elements of the tensor
$\Lambda$ are evaluated at the SBP nodes and stored in the diagonal matrices
$\Lamxx$, $\Lamxy$, $\Lamyx$ and $\Lamyy$.  For example,
\begin{equation*}
  \Lamxx = \mydiag\left(\lambda_{xx}(x_1,y_1), \lambda_{xx}(x_2,y_2),\ldots, \lambda_{xx}(x_n,y_n)\right).
\end{equation*}

The vectors $\bm{s}_{\kappa}^{\fnc{I}}$ and $\bm{s}_{\kappa}^{\fnc{B}}$ on the
right-hand side of \eqref{eq:parabolic_SBP} are the interface and boundary SAT
penalties, respectively.  For element $\kappa$ these penalties are defined by
\begin{equation*}
\bm{s}_{\kappa}^{\fnc{I}}\left(\bm{u}_{h}\right)
= \sum_{\gamma \subset \Gamma_\kappa^\fnc{I}}
\begin{bmatrix} \Rgk^T & \Dgk^T \end{bmatrix}
\setlength{\arraycolsep}{7pt}
\begin{bmatrix} \Siggk{1} & \Siggk{3} \\ \Siggk{2} & \Siggk{4} \end{bmatrix}
\begin{bmatrix} \Rgk\Ukp - \Rgn\Unu \\ \Dgk\Ukp + \Dgn\Unu \end{bmatrix} 
\end{equation*}
and
\begin{multline*}
  \bm{s}_{\kappa}^{\fnc{B}}\left(\bm{u}_{h}, \bm{u}_\fnc{D}, \bm{u}_\fnc{N}\right)
  = \sum_{\gamma \subset \Gamma_\kappa^{\fnc{D}}}
  \begin{bmatrix} \Rgk^T & \Dgk^T \end{bmatrix}
  \begin{bmatrix} \phnt{-}\Sig_{\gamma}^{\fnc{D}} \\ -\Bg \end{bmatrix}
  (\Rgk \Ukp -\bm{u}_{\gamma {\fnc{D}}}) \\
  + \sum_{\gamma \subset \Gamma_\kappa^{\fnc{N}}} \Rgk^{T}\B_{\gamma} (\Dgk \Ukp - \bm{u}_{\gamma {\fnc{N}}}),
\end{multline*}
respectively.  The set $\Gamma_{\kappa} = \partial \Omega_\kappa$ is the
boundary of element $\Omega_\kappa$, while $\Gamma_{\kappa}^{\fnc{D}} =
\Gamma_{\kappa} \cap \Gamma^\fnc{D}$ and $\Gamma_{\kappa}^{\fnc{N}} =
\Gamma_{\kappa} \cap \Gamma^\fnc{N}$. We use $\nu$ as the generic index of the
element sharing face $\gamma$ with the $\kappa$th element, \ie, $\gamma =
\Omega_\kappa \cap \Omega_\nu$.  The vectors $\bm{u}_{\gamma\fnc{D}}$ and
$\bm{u}_{\gamma\fnc{N}}$ in the boundary penalties denote the functions
$\fnc{U}_\fnc{D}$ and $\fnc{U}_\fnc{N}$, respectively, evaluated at the cubature
nodes of face $\gamma$.

Recall that the matrix $\Rgk \in \mathbb{R}^{n_{\gamma} \times \nk}$ is the
interpolation/extrapolation operator from the nodes of $\Omega_\kappa$ to the
nodes of face $\gamma \subset \Gamma_{\kappa}$, while $\Rgn \in
\mathbb{R}^{n_{\gamma} \times n_{\nu}}$ is the interpolation/extrapolation
operator from the nodes of the neighbor to the nodes of $\gamma$.  Therefore,
the normal derivative operators that discretize $\vec{n}\cdot(\Lambda\nabla)$ on
face $\gamma$ are given by
\begin{align*}
  \D_{\gamma\kappa} &= \Nxg \Rgk \left(\Lamxx \Dx + \Lamxy \Dy\right)_\kappa
  + \Nyg \Rgk \left(\Lamyx\Dx + \Lamyy\Dy\right)_\kappa, \\
  \D_{\gamma\nu} &= -\Nxg \Rgn \left(\Lamxx \Dx + \Lamxy \Dy\right)_{\nu}
  - \Nyg \Rgn \left(\Lamyx\Dx + \Lamyy\Dy\right)_{\nu}.
\end{align*}
In addition, recall that $\Nxg$ (\resp $\Nyg$) is a
diagonal matrix holding the $x$ (\resp $y$) component of the unit outward
normal, with respect to $\kappa$, at the cubature nodes of face $\gamma$.  Thus,
the sign of this matrix must be reversed for $\Dgn$.

Finally, the matrices $\Siggk{i} = \left(\Siggk{i}\right)^T \in \mathbb{R}^{n_{\gamma}
  \times n_{\gamma}}$, $i=1,2,3,4$ denote the symmetric SAT coefficient matrices
for element $\kappa$ on face $\gamma$.  Similarly, $\Sig_{\gamma}^{\fnc{D}}$ is the
coefficient matrix for the SAT on a Dirichlet boundary face of $\kappa$.  These
coefficients are to be determined in the following analysis.  Note that
$\Sig_{\gamma\kappa}^{(i)} \ne \Sig_{\nu\gamma}^{(i)}$ in general; that is, we
do not assume \emph{ab initio} that the coefficient matrices of two adjacent
elements are necessarily equal.

\ignore{
\begin{remark} 
  We assume the same SBP operator is used on all elements and that each face uses
  the same number of cubature nodes, although neither of these assumptions is not
  strictly necessary.
\end{remark}
}

\subsection{Face-based weak forms of the discretization}

The discretization \eqref{eq:parabolic_SBP} is the element-based strong form.
For the subsequent analysis, two equivalent face-based weak forms will prove
more useful.  Before deriving these weak formulations, we introduce two
identities that will be helpful.

\begin{proposition}\label{prop:identities}
  Let $\D_\kappa$ be defined as in~\eqref{eq:Dkappa}.  Then, $\forall\; \Ukp,
  \Vkp \in \mathbb{R}^{\nk}$, 
  \begin{align}
    \Vkp^T \Hk \D_\kappa \Ukp &= -\Vkp^T \Mk \Ukp
    + \sum_{\gamma \subset \Gamma_\kappa} \Vkp^T \Rgk^T \B_{\gamma} \Dgk \Ukp, 
    \label{eq:identity-1} \\
    \text{and}\qquad
    -\Vkp^T \Mk \Ukp &= 
    \Vkp^T \D_\kappa^T \Hk \Ukp 
    - \sum_{\gamma \subset \Gamma_\kappa} \Vkp^T \Dgk^T \B_{\gamma} \Rgk \Ukp,
    \label{eq:identity} \\
    \intertext{where $\Mk$ is the symmetric semi-definite matrix}
    \Mk &= \begin{bmatrix} \Dx \\ \Dy \end{bmatrix}_{\kappa}^T
    \begin{bmatrix}
      \H\Lamxx & \H\Lamxy \\
      \H\Lamyx & \H\Lamyy
    \end{bmatrix}_{\kappa}
    \begin{bmatrix}
      \Dx \\ \Dy
    \end{bmatrix}_{\kappa}. \notag 
  \end{align}
\end{proposition}

The proof of Proposition~\ref{prop:identities} is a straightforward application
of the properties of SBP operators and is omitted.

\begin{remark}
  The identities in Proposition~\ref{prop:identities} are the SBP analogs of
  applying integration by parts to $\int_{\Omega_\kappa} \fnc{V} \nabla \cdot(\Lambda
  \nabla \fnc{U}) \,\diff \Omega$ once (the first identity) and twice (the second
  identity).
\end{remark}

To obtain the element-based weak formulation, we first left multiply
\eqref{eq:parabolic_SBP} by $\Vkp^{T} \Hk$, where $\Vkp \in
\mathbb{R}^{\nk}$ is an arbitrary vector, and then apply \eqref{eq:identity-1}.
This produces the following form of the discretization: for all
$\kappa=1,2,\dots,K$, find $\Ukp \in \mathbb{R}^{\nk}$ such that, $\forall \Vkp
\in \mathbb{R}^{\nk}$,
\begin{multline*}
  \Vkp^{T} \Hk \frac{\diff \bm{u}_{\kappa}}{\diff t} = 
  -\Vkp^T \Mk \Ukp
  + \sum_{\gamma \subset \Gamma_\kappa} \Vkp^T \Rgk^T \B_{\gamma} \Dgk \Ukp
  + \Vkp^{T} \Hk \bm{f}_\kappa \\
  - \Vkp^{T} \bm{s}_{\kappa}^{\fnc{I}}\left(\bm{u}_{h}\right) 
  - \Vkp^{T} \bm{s}_{\kappa}^{\fnc{B}}\left(\bm{u}_{h}, \bm{u}_\fnc{D}, \bm{u}_\fnc{N}\right).
\end{multline*}

To obtain the first of two face-based weak formulations, we sum the
element-based weak form over all $\kappa$.  After rearrangement, this gives the
statement: find $\bm{u}_{h} \in \mathbb{R}^{\sum \nk}$ such that
\begin{equation*}
  \sum_{\kappa \in \fnc{T}_{h}} \Vkp^{T} \Hk \frac{\diff \bm{u}_{\kappa}}{\diff t}
  = B_{h}(\bm{u}_{h}, \bm{v}_{h}),
  \qquad \forall\; \bm{v}_{h} \in \mathbb{R}^{\sum \nk},
\end{equation*}
where the bilinear form on the right is defined by
\begin{multline} \label{face-based-form-1}
  B_{h}(\bm{u}_h, \bm{v}_h) := - \sum_{\kappa \in \fnc{T}_{h}} \Vkp^T
  \Mk \Ukp + \sum_{\kappa \in \fnc{T}_{h}} \Vkp^{T} \Hk 
  \bm{f}_\kappa \\ - \sum_{\gamma \subset \Gamma^{\fnc{I}}}
  \begin{bmatrix} \Rgk \Vkp \\ \Rgn \Vnu \\ \Dgk \Vkp \\ \Dgn \Vnu \end{bmatrix}^T
  \setlength{\arraycolsep}{7pt}
  \begin{bmatrix}
    \phnt{-}\Siggk{1} & -\Siggk{1} & \Siggk{3} - \Bg & \Siggk{3} \\
    -\Siggn{1} & \phnt{-}\Siggn{1} & \Siggn{3} & \Siggn{3} - \Bg \\
    \phnt{-}\Siggk{2} & -\Siggk{2} & \Siggk{4} & \Siggk{4} \\
    -\Siggn{2} & \phnt{-}\Siggn{2} & \Siggn{4} & \Siggn{4}
  \end{bmatrix}
  \begin{bmatrix} \Rgk \Ukp \\ \Rgn \Unu \\ \Dgk \Ukp \\ \Dgn \Unu \end{bmatrix} \\
  - \sum_{\gamma \subset \Gamma^\fnc{D}}
  \begin{bmatrix} \Rgk \Vkp \\ \Dgk \Vkp \end{bmatrix}^T
  \begin{bmatrix}
    \phnt{-}\Sig_{\gamma}^{\fnc{D}} & -\Bg \\ -\Bg & \mat{0}
  \end{bmatrix} 
  \begin{bmatrix} \Rgk \Ukp - \bm{u}_{\gamma\fnc{D}} \\ \Dgk \Ukp \end{bmatrix}
  + \sum_{\gamma \subset \Gamma^\fnc{N}} \Vkp^T \Rgk^T \Bg \bm{u}_{\gamma\fnc{N}}.
\end{multline}
The bilinear form \eqref{face-based-form-1} will be our starting point in the
energy stability analysis presented later.

An equivalent face-based bilinear form, which will be useful for the
adjoint analysis, is obtained by using \eqref{eq:identity} in
\eqref{face-based-form-1}.  This produces
\begin{multline} \label{face-based-form-2}
  B_{h}(\bm{u}_h,\bm{v}_h) 
  \equiv \sum_{\kappa \in \fnc{T}_{h}} \Vkp^T \D_{\kappa}^T \Hk \Ukp 
  + \sum_{\kappa \in \fnc{T}_{h}} \Vkp^{T} \Hk \bm{f}_\kappa 
  + \sum_{\gamma \subset \Gamma^\fnc{D}} \Vkp^T \Dgk^T \Bg \bm{u}_{\gamma \fnc{D}}
\\
  - \sum_{\gamma \subset \Gamma^{\fnc{I}}}
  \begin{bmatrix} \Rgk \Vkp \\ \Rgn \Vnu \\ \Dgk \Vkp \\ \Dgn \Vnu \end{bmatrix}^T
  \setlength{\arraycolsep}{5pt}
  \begin{bmatrix}
    \phnt{-}\Siggk{1} & -\Siggk{1} & \Siggk{3} - \Bg & \Siggk{3} \\
    -\Siggn{1} & \phnt{-}\Siggn{1} & \Siggn{3} & \Siggn{3} - \Bg \\
    \phnt{-}\Siggk{2} + \Bg & -\Siggk{2} & \Siggk{4} & \Siggk{4} \\
    -\Siggn{2} & \phnt{-}\Siggn{2} + \Bg & \Siggn{4} & \Siggn{4}
  \end{bmatrix}
  \begin{bmatrix} \Rgk \Ukp \\ \Rgn \Unu \\ \Dgk \Ukp \\ \Dgn \Unu \end{bmatrix} \\
  - \sum_{\gamma \subset \Gamma^\fnc{D}}
  \begin{bmatrix} \Rgk \Vkp \end{bmatrix}^T
  \begin{bmatrix}
    \phnt{-}\Sig_{\gamma}^\fnc{D} & -\Bg \end{bmatrix}
  \begin{bmatrix} \Rgk \Ukp - \bm{u}_{\gamma\fnc{D}} \\ \Dgk \Ukp \end{bmatrix}
  + \sum_{\gamma \subset \Gamma^\fnc{N}} 
  \begin{bmatrix} \Rgk \Vkp \\ \Dgk \Vkp \end{bmatrix}^T
  \begin{bmatrix} \Bg \bm{u}_{\gamma\fnc{N}} \\ -\Bg \Rgk \Ukp \end{bmatrix}.
\end{multline}

\section{Adjoint consistency analysis}\label{sec:adjoint}

It is well known in the finite-element community that adjoint, or dual,
consistency is necessary for obtaining optimal error rates in the $L^2$
norm~\cite{Arnold2002}.  More generally, adjoint consistency leads to
superconvergent (integral) functional estimates~\cite{Babuska1984post,
  Babuska1984postpart2, pierce:2000, giles:2002, lu:2005, hartmann:2007b,
  fidkowski:2011}, which can significantly improve the accuracy of outputs like
lift and drag when using high-order methods.  Given the close connection between
SBP finite-difference methods and the FE methods, it is perhaps not surprising that
classical (\ie tensor-product) SBP discretizations also exhibit superconvergent
functionals when discretized in a dual consistent
manner~\cite{Hicken2011superconvergent, Hicken2014dual}.

For the reasons listed above, adjoint consistency is a property that we would
like our multi-dimensional SBP discretizations to satisfy.  Therefore, in the
following sections, we investigate the constraints on the SAT penalties in
\eqref{eq:parabolic_SBP} that guarantee adjoint consistency.  We begin by
briefly reviewing the dual problem associated with steady version of
\eqref{eq:parabolic}.

\subsection{A generic adjoint PDE}

An adjoint is defined by the primal PDE and a particular functional of interest.
For the following adjoint-consistency analysis, we consider the linear
functional
\begin{equation}\label{eq:fun}
  \fnc{J}(\fnc{U}) = \int_{\Omega} \fnc{G} \fnc{U} \, \diff \Omega
   +  \int_{\Gamma^\fnc{N}} \fnc{V}_\fnc{N} \fnc{U} \, \diff\Gamma  \\
  -  \int_{\Gamma^\fnc{D}} \fnc{V}_\fnc{D} \hat{\bm{n}} \cdot \left(\Lambda \nabla \fnc{U} \right) \, \diff\Gamma ,
\end{equation}
where $\fnc{G} \in L^{2}(\Omega)$, $\fnc{V}_\fnc{D} \in L^{2}(\Gamma^\fnc{D})$
and $\fnc{V}_\fnc{N} \in L^2(\Gamma^\fnc{N})$.  One can show that the adjoint
PDE corresponding to the steady form of \eqref{eq:parabolic} and \eqref{eq:fun}
is~\cite{Lanczos1961linear}
\begin{equation}\label{eq:adjoint}
  \begin{aligned}
    -&\nabla\cdot\left(\Lambda \nabla\fnc{V} \right) = \fnc{G},\quad &&\forall \; x \in \Omega \\
    &\fnc{V} = \fnc{V}_\fnc{D}, \quad &&\forall \; x \in \Gamma^\fnc{D}  \\
    &\hat{\bm{n}} \cdot \left( \Lambda \nabla \fnc{V}\right) = \fnc{V}_\fnc{N},
    &&\forall \; x \in \Gamma^\fnc{N}.
  \end{aligned}
\end{equation}

\subsection{Functional and adjoint discretization}

We discretize the functional \eqref{eq:fun} as
\begin{multline}\label{eq:fun_SBP}
J_h(\bm{u}_h) := \sum_{\Omega_\kappa \in \fnc{T}_{h}}\bm{g}_{\kappa}^T\Hk\Ukp
    + \sum_{\gamma \subset \Gamma^{\fnc{N}}} \bm{v}_{\gamma \fnc{N}}^T\Bg\Rgk\Ukp
    - \sum_{\gamma \subset \Gamma^{\fnc{D}}} \bm{v}_{\gamma \fnc{D}}^T\Bg\Dgk\Ukp \\
    + \sum_{\gamma \subset \Gamma^{\fnc{D}}} \bm{v}_{\gamma \fnc{D}}^T\Sig_{\gamma}^{\fnc{D}} (\Rgk\Ukp - \bm{u}_{\gamma \fnc{D}}),
\end{multline}
where $\bm{v}_{\gamma \fnc{N}}$ and $\bm{v}_{\gamma \fnc{D}}$ denote
$\fnc{V}_{\fnc{N}}$ and $\fnc{V}_{\fnc{D}}$, respectively, evaluated at the
cubature nodes of the generic face $\gamma$, and
\begin{equation}
\bm{g}_{\kappa}^T = [\fnc{G}(x_0)\  \fnc{G}(x_1)\ \dots\  \fnc{G}(x_{\nk})].
\end{equation}

\begin{remark}
  The first three terms in \eqref{eq:fun_SBP} are direct discretizations of the
  first three terms in \eqref{eq:fun}.  The fourth term in \eqref{eq:fun_SBP} is
  an order $h^{r+1}$ term; the interpolation/extrapolation operators are exact
  for degree $r \geq p$ polynomials, so $\Rgk \Ukp = \bm{u}_{\gamma \fnc{D}} +
  \text{O}(h^{r+1})$.  This last term in $J_h$ is necessary for adjoint
  consistency.
\end{remark}

The discrete adjoint equation is defined implicitly based on $J_h$ and the
discretization~\eqref{eq:parabolic_SBP}.  Specifically, to find the adjoint
equation we form the discrete Lagrangian, take its first variation with respect
to $\bm{u}_{h}$, and set the result to zero.  To this end, we add the face-based
weak form \eqref{face-based-form-2} to $J_h$ and, after some algebraic
manipulation\footnote{In particular, note that the functional and bilinear form
  are scalars, so $J_{h}(\bm{u}_h)^T = J_{h}(\bm{u}_h)$ and
  $B_{h}(\bm{u}_h,\bm{v}_h)^T = B_{h}(\bm{u}_h,\bm{v}_h)$.}, we get the
Lagrangian
\begin{equation*}
  L_{h}(\bm{u}_h, \bm{v}_h) = J_{h}(\bm{u}_h) + B_h(\bm{u}_h, \bm{v}_h)
  = J_{h}^{*}(\bm{v}_h) + B_h^{*}(\bm{v}_h, \bm{u}_h),
\end{equation*}
where the dual form of the functional is defined by
\begin{multline*}
  J_{h}^{*}(\bm{v}_h) 
  = \sum_{\kappa \in \fnc{T}_{h}} \Vkp^{T} \Hk \bm{f}_\kappa 
  + \sum_{\gamma \subset \Gamma^\fnc{N}} \bm{u}_{\gamma \fnc{N}}^T \Bg \Rgk \Vkp
  + \sum_{\gamma \subset \Gamma^\fnc{D}} \bm{u}_{\gamma \fnc{D}}^T \Bg \Dgk \Vkp \\
  + \sum_{\gamma \subset \Gamma^\fnc{D}} \bm{u}_{\gamma \fnc{D}}^T \Sig_{\gamma}^\fnc{D}
  (\Rgk \Vkp - \bm{v}_{\gamma \fnc{D}}),
\end{multline*}
and the adjoint bilinear form is given by
\begin{multline*}
B_{h}^{*}(\bm{v}_h, \bm{u}_h)
= \sum_{\kappa \in \fnc{T}_{h}} \Ukp^T \Hk \D_{\kappa} \Vkp
+ \sum_{\Omega_\kappa \in \fnc{T}_{h}} \Ukp^T \Hk \bm{g}_{\kappa}
   \\
  - \sum_{\gamma \subset \Gamma^{\fnc{I}}}
  \begin{bmatrix} \Rgk \Ukp \\ \Rgn \Unu \\ \Dgk \Ukp \\ \Dgn \Unu \end{bmatrix}^T
  \setlength{\arraycolsep}{5pt}
  \begin{bmatrix}
    \phnt{-}\Siggk{1}       & -\Siggn{1}      & \Siggk{2} + \Bg & -\Siggn{2} \\
    -\Siggk{1}      & \phnt{-}\Siggn{1}       & -\Siggk{2}      & \Siggn{2} + \Bg \\
    \Siggk{3} - \Bg & \phnt{-}\Siggn{3}       & \phnt{-}\Siggk{4}       & \phnt{-}\Siggn{4} \\
    \phnt{-}\Siggk{3}       & \Siggn{3} - \Bg & \phnt{-}\Siggk{4}  & \phnt{-}\Siggn{4} 
  \end{bmatrix}
  \begin{bmatrix} \Rgk \Vkp \\ \Rgn \Vnu \\ \Dgk \Vkp \\ \Dgn \Vnu \end{bmatrix} \\
  - \sum_{\gamma \subset \Gamma^\fnc{D}}
  \begin{bmatrix} \Rgk \Ukp \\ \Dgk \Ukp \end{bmatrix}^T
  \begin{bmatrix}
    \phnt{-}\Sig_{\gamma}^{\fnc{D}} \\ -\Bg \end{bmatrix}
  (\Rgk \Vkp - \bm{v}_{\gamma \fnc{D}})
  - \sum_{\gamma \subset \Gamma^\fnc{N}} \Ukp^T \Rgk^T \Bg(\Dgk \Vkp - \bm{v}_{\gamma \fnc{N}}).
\end{multline*}

As already mentioned, the discrete adjoint equation is found by setting the
first variation of $L_{h}(\bm{u}_h,\bm{v}_h)$ with respect to $\bm{u}_h$ to zero.
Since the primal variable is finite dimensional here, taking the first variation
is equivalent to finding the gradient of $L_{h}$ with respect to $\bm{u}_h$.
Furthermore, we see that $J_{h}^{*}(\bm{v}_h)$ does not depend on $\bm{u}_{h}$,
so we only need to consider the gradient of $B_{h}^{*}(\bm{v}_h,\bm{u}_{h})$.

Taking the gradient of $B_{h}^{*}(\bm{v}_h,\bm{u}_{h})$ with respect to $\Ukp$,
multiplying by $\Hk^{-1}$, and setting the result to zero (\ie setting
the first variation to zero), gives the following element-based strong form of
the adjoint equation:
\begin{equation}\label{eq:adjoint_SBP}
  \Hk^{-1} \frac{\partial B_{h}^{*}}{\partial \Ukp} =
  \D_{\kappa} \Vkp + \bm{g}_{\kappa}
  - \Hk^{-1} (\bm{s}_{\kappa}^{\fnc{I}})^{*}(\bm{v}_h)
  - \Hk^{-1} (\bm{s}_{\kappa}^{\fnc{B}})^{*}(\bm{v}_h,\bm{v}_\fnc{D},\bm{v}_\fnc{N})
  = \bm{0},
\end{equation}
where the adjoint SAT penalties for the interfaces are 
\begin{equation} \label{ad_interface_sat}
  (\bm{s}_{\kappa}^{\fnc{I}})^{*}\left(\bm{v}_{h}\right)
  = \sum_{\gamma \subset \Gamma_\kappa^\fnc{I}}
  \begin{bmatrix} \Rgk^T & \Dgk^T \end{bmatrix}
  \setlength{\arraycolsep}{7pt}
  \setlength{\arraycolsep}{5pt}
  \begin{bmatrix}
    \phnt{-}\Siggk{1}       & -\Siggn{1}      & \Siggk{2} + \Bg & -\Siggn{2} \\
    \Siggk{3} - \Bg & \phnt{-}\Siggn{3}       & \phnt{-}\Siggk{4}       & \phnt{-}\Siggn{4}
  \end{bmatrix}
  \begin{bmatrix} \Rgk \Vkp \\ \Rgn \Vnu \\ \Dgk \Vkp \\ \Dgn \Vnu \end{bmatrix} 
\end{equation}
and the penalties for the boundaries are
\begin{multline*}
  (\bm{s}_{\kappa}^{\fnc{B}})^{*}\left(\bm{u}_{h}, \bm{u}_\fnc{D}, \bm{u}_\fnc{N}\right)
  = \sum_{\gamma \subset \Gamma_\kappa^{\fnc{D}}}
  \begin{bmatrix} \Rgk^T & \Dgk^T \end{bmatrix}
  \begin{bmatrix} \phnt{-}\Sig_{\gamma}^{\fnc{D}} \\ -\Bg \end{bmatrix}
  (\Rgk \Vkp - \bm{v}_{\gamma\fnc{D}}) \\
  + \sum_{\gamma \subset \Gamma_\kappa^{\fnc{N}}} \Rgk^{T}\B_{\gamma} (\Dgk \Vkp - \bm{v}_{\gamma\fnc{N}}).
\end{multline*}

\subsection{Adjoint consistency of the interface SAT}

The sum $\D_{\kappa} \Vkp + \bm{g}_{\kappa}$ in \eqref{eq:adjoint_SBP} is an
order $h^{p+1}$ discretization of the adjoint PDE in \eqref{eq:adjoint}.
Indeed, $\D_{\kappa}$ is the same operator used in the primal discretization.
Furthermore, the boundary SAT, $(\bm{s}_{\kappa}^{\fnc{B}})^{*}$, introduces an
error that is also $\text{O}(h^{p+1})$.  To see this, recall that $\Rgk$ and
$\Dgk$ are exact for polynomials of degree $p$, and $\bm{v}_{\gamma\fnc{D}}$ and
$\bm{v}_{\gamma\fnc{N}}$ are the exact boundary values evaluated at the nodes of
$\gamma$.  Thus, the differences $\Rgk \Vkp - \bm{v}_{\gamma\fnc{D}}$ and $\Dgk \Vkp
- \bm{v}_{\gamma\fnc{N}}$ vanish for polynomial solutions of degree $p$ or less.
Only the interface SATs require further scrutiny to determine adjoint
consistency.

\begin{theorem} \label{thm:adjoint_conditions}
  The primal discretization \eqref{eq:parabolic_SBP} and functional
  discretization \eqref{eq:fun_SBP} are adjoint consistent of order $h^{p+1}$
  provided the exact adjoint $\fnc{V}$ is sufficiently smooth on $\Omega$ and
  the SAT penalty matrices satisfy
  \begin{equation} \label{conditions}
    \begin{alignedat}{2}
      \Siggk{1} &= \Siggn{1}, &\qquad  
      \Siggk{2} + \Siggn{2} &= -\B_{\gamma}, \\
      \Siggk{3} + \Siggn{3} &= \B_{\gamma}, &\qquad
      \Siggk{4} &= \Siggn{4}.
    \end{alignedat}
  \end{equation}
\end{theorem}

\begin{proof}
  We have already considered the discretization of the spatial derivatives and
  the boundary SATs.  To show that the interface SAT is order $h^{p+1}$, it is
  sufficient to show that $(\bm{s}_{\kappa}^{\fnc{I}})^{*}\left(\bm{v}_{h}\right) =
  \bm{0}$ for polynomial solutions $\fnc{V} \in \poly{p}(\Omega_\kappa)$.  For
  these polynomials, the interpolation/extrapolation and normal-derivative
  operators are exact and we have
  \begin{equation*}
    \Rgk \Vkp = \Rgn \Vnu \equiv \bm{v}_{\gamma},
    \qquad\text{and}\qquad
    \Dgk \Vkp = -\Dgn \Vnu \equiv \bm{v}_{\gamma}'.
  \end{equation*}
  Substituting these identities into the adjoint interface SATs (\ref{ad_interface_sat}) gives
  \begin{equation*}
  (\bm{s}_{\kappa}^{\fnc{I}})^{*}\left(\bm{v}_{h}\right)
  = \sum_{\gamma \subset \Gamma_\kappa^\fnc{I}}
  \begin{bmatrix} \Rgk^T & \Dgk^T \end{bmatrix}
  \setlength{\arraycolsep}{7pt}
  \setlength{\arraycolsep}{5pt}
  \begin{bmatrix}
    \phnt{-}\Siggk{1} -\Siggn{1}   & \Siggk{2} + \Siggn{2} + \Bg \\
    \Siggk{3} + \Siggn{3} - \Bg    & \Siggk{4} - \Siggn{4}
  \end{bmatrix}
  \begin{bmatrix} \bm{v}_\gamma \\ \bm{v}_\gamma' \end{bmatrix}.
  \end{equation*}
  Thus, the adjoint interface penalties vanishes under the conditions
  \eqref{conditions}.\qed
\end{proof}

\begin{remark}
  It is straightforward to show that the conditions~\eqref{conditions} also imply that
  the SBP-SAT discretization is locally conservative, in the sense that
  $\sum_{\kappa \in \fnc{T}_{h}'} \bm{1}^T \Hk \diff \Ukp/\diff t$ depends only
  on the boundary faces of $\fnc{T}_{h}'$ when $\bm{f}_{\kappa} =\bm{0}$, for any
  subset of elements $\fnc{T}_{h}' \subset \fnc{T}_{h}$.
\end{remark}

\section{Energy analysis} \label{sec:energy_analysis}


\subsection{Energy analysis of the PDE}

\ignore{
The $L^2$ norm of the solution to \eqref{eq:parabolic} is bounded by the
problem data.  To see this, multiply the PDE by $\fnc{U}$ and integrate over
the domain $\Omega$.  After some manipulation one finds 
\begin{multline*}
  \frac{1}{2} \frac{\diff}{\diff t} \int_{\Omega} \fnc{U}^{2} \, \diff \Omega
  = - \int_{\Omega} \left(\nabla \fnc{U}\right) \cdot \Lambda \left(\nabla \fnc{U} \right) \, \diff \Omega
  + \int_{\Omega} \fnc{U} \fnc{F} \, \diff \Omega \\
  - \int_{\Gamma^\fnc{D}} \fnc{U}_\fnc{D} \hat{\bm{n}} \cdot \left(\Lambda \nabla \fnc{U}\right) \, \diff \Gamma - \int_{\Gamma^\fnc{N}} \fnc{U} \fnc{U}_\fnc{N} \, \diff \Gamma.
\end{multline*}
It is easy to see that the integrand of the first term on the right is
non-negative, so this term leads to a decay in $\|\fnc{U}\|$.  The remaining
terms may produce growth or decay.
}

The difference between two solutions of \eqref{eq:parabolic}, $\fnc{W} = \fnc{U}
- \fnc{V}$, satisfies the homogeneous PDE with $\fnc{F} = 0$, $\fnc{U}_\fnc{D} =
0$, $\fnc{U}_\fnc{N} = 0$, and $\fnc{U}_{0} = 0$.  It follows that $\fnc{W}$
satisfies
\begin{equation*}
  \frac{1}{2} \frac{\diff}{\diff t} \int_{\Omega} \fnc{W}^{2} \, \diff \Omega =
  -\int_{\Omega} \left(\nabla \fnc{W}\right) \cdot \Lambda \left(\nabla
  \fnc{W} \right) \, \diff \Omega \leq 0,
\end{equation*}
and, since $\fnc{W} = 0$ initially, $\fnc{W} = 0$ for all time.  This well-known
proof shows that solutions to \eqref{eq:parabolic} are unique.  We would like to
mimic this proof in the case of the discretization \eqref{eq:parabolic_SBP}.
The following section investigates the conditions on the SAT penalties that make
this possible.

\subsection{Energy analysis of the discrete homogeneous problem}

The objective of this section is to further constrain the SAT penalty matrices
based on the conditions for discrete energy stability. Before presenting the
conditions for energy stability, we first simplify the penalty matrices based on
the adjoint consistency conditions \eqref{conditions}.  In particular, we will
drop the dependence of the $\Sig^{(1)}$ and $\Sig^{(4)}$ matrices on the
elements:
\begin{equation*}
  \Siggk{1} = \Siggn{1} \equiv \Siggam{1},
  \qquad\text{and}\qquad
  \Siggk{4} = \Siggn{4} \equiv \Siggam{4}.
\end{equation*}
In addition, we will also assume that
\begin{equation} \label{eqn:assumption_on_sigma2}
\Siggk{3} -\Siggk{2}=\Bg.
 \end{equation}
This is not strictly required by the adjoint-consistency analysis, but by the
desire to make the 4x4 block matrix in (\ref{face-based-form-1}) symmetric,
which significantly simplifies the energy analysis; note that the above
condition together with the conditions of Theorem \ref{thm:adjoint_conditions} imply
that $\Siggk{3} = -\Siggn{2}$, $\Siggn{3} = -\Siggk{2}$, and $\Siggn{3} -
\Siggn{2} = \Bg$, which, in turn, imply the symmetry of the 4x4 block matrix in
(\ref{face-based-form-1}).  \ignore{ Based on this assumption and the conditions
  in \eqref{conditions} we have
\begin{equation*}
  \Siggk{2} = \Siggn{2} = -\frac{1}{2}\Bg,
  \qquad\text{and}\qquad
  \Siggk{3} = \Siggn{3} = \frac{1}{2}\Bg.
\end{equation*}
In addition to simplifying the stability analysis, our motivation for this
assumption is that $\Sig^{(2)}$ and $\Sig^{(3)}$ do not help control the
coercivity of the bilinear form.
}

We will need the following lemma for the stability analysis.  The purpose of the
lemma is to shift the volume terms in the bilinear form $B_h$ to the faces, so
that these terms can contribute to the semi-definiteness of the interface terms.

\begin{lemma} \label{lem:energy_start}
  For each face $\gamma$ of element $\kappa$, let $\alpha_{\gamma \kappa} \ge 0$
  such that $\sum_{\gamma \subset \Gamma_{\kappa}} \alpha_{\gamma
    \kappa} = 1$.  Then, the bilinear form corresponding to the SBP-SAT
  discretization of the homogeneous version of the PDE~\eqref{eq:parabolic} can
  be written as
  \begin{multline}\label{face-based-form-3}
    B_{h}(\bm{u}_h,\bm{v}_h) = \\
    - \sum_{\gamma \subset \Gamma^{\fnc{I}}}
  \begin{bmatrix} \Rgk \Vkp \\ \Rgn \Vnu \\ \Gk \Vkp \\ \Gn \Vnu
  \end{bmatrix}^{T} 
  \setlength{\arraycolsep}{5pt}
  \begin{bmatrix} 
    \phnt{-}\Siggam{1} & -\Siggam{1} & \phnt{-}\Siggk{2} \Cgk & -\Siggn{2}\Cgn \\
    -\Siggam{1} & \phnt{-}\Siggam{1} & -\Siggk{2}\Cgk & \phnt{-}\Siggn{2}\Cgn \\
    \phnt{-}\Cgk^T\Siggk{2} & -\Cgk^{T}\Siggk{2} & \alpha_{\gamma \kappa}\Lam_{\kappa}^{*} &  \\
    -\Cgn^T\Siggn{2} & \phnt{-}\Cgn^{T}\Siggn{2} &  & \alpha_{\gamma \nu} \Lam_{\nu}^{*}   
  \end{bmatrix}
  \begin{bmatrix} \Rgk \Ukp \\ \Rgn \Unu \\ \Gk \Ukp \\ \Gn \Unu
  \end{bmatrix} \\
  - \sum_{\gamma \subset \Gamma^{\fnc{I}}}
  \begin{bmatrix} \Dgk \Vkp \\ \Dgn \Vnu \end{bmatrix}^T
  \begin{bmatrix} \Siggam{4} & \Siggam{4} \\ \Siggam{4} & \Siggam{4} \end{bmatrix}
  \begin{bmatrix} \Dgk \Ukp \\ \Dgn \Unu \end{bmatrix} \\
  - \sum_{\gamma \subset \Gamma^{\fnc{D}}}
  \begin{bmatrix} \Rgk \Vkp \\ \Gk \Vkp \end{bmatrix}^{T}
  \setlength{\arraycolsep}{5pt}
  \begin{bmatrix} \phnt{-}\Sig_{\gamma}^{\fnc{D}} & -\Bg \Cgk \\
    -\Cgk^T \Bg & \alpha_{\gamma \kappa}\Lam_{\kappa}^{*} \end{bmatrix}
  \begin{bmatrix} \Rgk \Ukp \\ \Gk \Ukp \end{bmatrix},
  \end{multline}
  where we have introduced the matrices
  \begin{alignat*}{2}
    \Gk &= \left\{ \begin{bmatrix} \Lamxx & \Lamxy \\ \Lamyx &
      \Lamyy \end{bmatrix} \begin{bmatrix} \Dx \\ \Dy \end{bmatrix} \right\}_{\kappa},
    &
    \setlength{\arraycolsep}{5pt}    
    \Cgk &= \begin{bmatrix} \Nxg \Rgk & \Nyg \Rgk \end{bmatrix}, \\
    \Gn &= \left\{ \begin{bmatrix} \Lamxx & \Lamxy \\ \Lamyx &
      \Lamyy \end{bmatrix} \begin{bmatrix} \Dx \\ \Dy \end{bmatrix} \right\}_{\nu},
    & 
    \setlength{\arraycolsep}{5pt}
    \Cgn &= - \begin{bmatrix} \Nxg \Rgn & \Nyg \Rgn \end{bmatrix}, \\
    \intertext{and}
    \Lam_{\kappa}^{*} &= \left\{ 
    \begin{bmatrix} \Lamxx & \Lamxy \\ \Lamyx & \Lamyy \end{bmatrix}^{-1}
    \begin{bmatrix} \H & \\ & \H \end{bmatrix} \right\}_{\kappa}, 
    & \qquad
    \Lam_{\nu}^{*} &= \left\{ 
    \begin{bmatrix} \Lamxx & \Lamxy \\ \Lamyx & \Lamyy \end{bmatrix}^{-1}
    \begin{bmatrix} \H & \\ & \H \end{bmatrix} \right\}_{\nu}
  \end{alignat*}
\end{lemma}

\begin{proof}
  The full proof follows from straightforward algebra and is omitted; however,
  we will highlight two observations that make the connection between
  \eqref{face-based-form-3} and \eqref{face-based-form-1} clearer. First, we note
  that
  \begin{equation*}
    \Cgk \Gk = \Dgk  \qquad\text{and}\qquad
    \Cgn \Gn = \Dgn.
  \end{equation*}
  Second, the elemental matrix $\Mk$ can be decomposed as
  \begin{equation*}
    \Mk = 
    \begin{bmatrix} \Dx \\ \Dy \end{bmatrix}_{\kappa}^T
    \begin{bmatrix}
      \H\Lamxx & \H\Lamxy \\
      \H\Lamyx & \H\Lamyy
    \end{bmatrix}_{\kappa}
    \begin{bmatrix}
      \Dx \\ \Dy
    \end{bmatrix}_{\kappa} =
    \sum_{\gamma \subset \Gamma_{\kappa}} \alpha_{\gamma \kappa}
    \Gk^T \Lam_{\kappa}^{*} \Gk.
  \end{equation*}
  \qed
\end{proof}

We will now state and prove the main energy-stability result.

\begin{theorem} \label{thm:stability-condition}
  The SBP-SAT discretization corresponding to the homogeneous version
  of~\eqref{eq:parabolic} has a non-increasing solution norm, with respect to
  the $\H$ matrix, provided 
  \begin{gather}
    \Siggam{1} - \Siggk{2} \Cgk
    \left(\alpha_{\gamma\kappa}\Lam_{\kappa}^{*}\right)^{-1}\Cgk^T \Siggk{2}
    - \Siggn{2} \Cgn \left(\alpha_{\gamma\nu} \Lam_{\nu}^{*}\right)^{-1}\Cgn^T \Siggn{2} 
    \succeq \mat{0}, \label{eq:Sig1_cond} \\
    \Sig_{\gamma}^{\fnc{D}} - \Bg \Cgk
    \left(\alpha_{\gamma\kappa}\Lam_{\kappa}^{*}\right)^{-1}\Cgk^T \Bg \succeq \mat{0},
    \label{eq:SigD_cond}
  \end{gather}
  and $\Siggam{4} \succeq \mat{0}$, where $\mat{A} \succeq \mat{0}$ indicates
  that $\mat{A}$ is positive semi-definite.
\end{theorem}

\begin{proof}
  The SBP-SAT discretization of the homogeneous equation is given by
  \begin{equation*}
    \sum_{\kappa \in \mathcal{T}_{h}} \Vkp^T \Hk \frac{\diff \Wkp}{\diff t}
    = B_{h}(\bm{w}_{h},\bm{v}_{h}),
  \end{equation*}
  where $\B_{h}(\bm{w}_{h},\bm{v}_{h})$ is defined in \eqref{face-based-form-3}.
  If we can show that $B_{h}(\bm{w}_{h},\bm{w}_{h}) \leq 0$ for all
  $\bm{w}_{h}$, then we will have $\sum_{\kappa \in \mathcal{T}_{h}} \Wkp^T
  \Hk \diff \Wkp/\diff t \leq 0$ and the desired result will follow.

  The scalar $B_{h}(\bm{w}_{h},\bm{w}_{h})$ is nonpositive if the symmetric
  matrices in the three sums of \eqref{face-based-form-3} are positive
  semi-definite.  We begin by considering the matrix that appears in the sum
  over the faces of the Dirichlet boundary:
  \begin{equation*}
    \setlength{\arraycolsep}{5pt}
    \begin{bmatrix} \phnt{-}\Sig_{\gamma}^{\fnc{D}} & -\Bg \Cgk \\
      -\Cgk^T \Bg & \alpha_{\gamma \kappa}\Lam_{\kappa}^{*} \end{bmatrix}
    \succeq \mat{0}.
  \end{equation*}
  Since, $\alpha_{\gamma\kappa}\Lam_{\kappa}^{*}$ is positive definite, the
  above matrix is positive semi-definite if the associated Schur complement is
  positive semi-definite:
  \begin{equation*}
    \Sig_{\gamma}^{\fnc{D}} - \Bg \Cgk 
    \left(\alpha_{\gamma\kappa}\Lam_{\kappa}^{*}\right)^{-1}
    \Cgk^T \Bg \succeq \mat{0},
  \end{equation*}
  which is precisely the condition~\eqref{eq:SigD_cond}.

  Next, consider the matrix involving $\Siggam{4}$ in \eqref{face-based-form-3}:
  \begin{equation*}
     \begin{bmatrix} \Siggam{4} & \Siggam{4} \\ \Siggam{4} & \Siggam{4} \end{bmatrix}
     = \begin{bmatrix} 1 & 1 \\ 1 & 1 \end{bmatrix} \otimes \Siggam{4},
  \end{equation*}
  where $\otimes$ denotes the Kronecker product.  Since the eigenvalues of
  $\left[\begin{smallmatrix} 1 & 1 \\ 1 & 1 \end{smallmatrix}\right]$ are zero
  and two, it follows from the spectral theory of Kronecker products that the
  eigenvalues of the above matrix are twice the eigenvalues of $\Siggam{4}$ and
  $n_{\gamma}$ zeros.  Thus, we require that $\Siggam{4} \succeq \mat{0}$.

  Finally, we analyze the matrix containing $\Siggam{1}$.  Similar to the matrix
  in the boundary-face sum, we make use of the fact that $\alpha_{\gamma\kappa}
  \Lam_{\kappa}^{*}$ and $\alpha_{\gamma\nu} \Lam_{\nu}^{*}$ are positive
  definite to conclude that the $4\times 4$ block matrix is positive
  semi-definite if the Schur complement is also positive semi-definite, \ie
  \begin{equation*}
    \begin{bmatrix} \phnt{-}1 & -1 \\ -1 & \phnt{-}1 \end{bmatrix}\otimes
    \left\{ 
    \Siggam{1} - 
    \begin{bmatrix} \Siggk{2} \Cgk & \Siggn{2} \Cgn \end{bmatrix}
    \begin{bmatrix} (\alpha_{\gamma\kappa}\Lam_{\kappa}^{*})^{-1} & \\
      & (\alpha_{\gamma\nu} \Lam_{\nu}^{*})^{-1} \end{bmatrix}
    \begin{bmatrix} \Cgk^T \Siggk{2} \\ \Cgn^T \Siggn{2} \end{bmatrix}
    \right\} \succeq \mat{0}.
  \end{equation*}
  The eigenvalues of $\left[\begin{smallmatrix} 1 & -1 \\ -1 &
      1 \end{smallmatrix}\right]$ are zero and two; thus, to ensure that the
  above Kronecker product is positive semi-definite, we must require that
  \begin{equation*}
    \Siggam{1} - \Siggk{2} \Cgk
    \left(\alpha_{\gamma\kappa}\Lam_{\kappa}^{*}\right)^{-1}\Cgk^T \Siggk{2}
    - \Siggn{2} \Cgn \left(\alpha_{\gamma\nu} \Lam_{\nu}^{*}\right)^{-1}\Cgn^T \Siggn{2} 
    \succeq \mat{0},
  \end{equation*}
  which is condition \eqref{eq:Sig1_cond}.\qed
\end{proof}

\ignore{
\todo[inline]{JEH: should we prove that $\Lam_{\kappa}^{*}$ is symmetric
  positive definite?~\cite{MEENAKSHI1999}}
}
\section{Generalization of existing methods}\label{sec:generalize}

In Sections \ref{sec:adjoint} and \ref{sec:energy_analysis} we obtained
sufficient conditions that ensure adjoint consistency and energy stability. In
this section we show that these conditions can be used to recover two popular
interior penalty methods used in FE methods, namely, the modified scheme of
Bassi and Rebay (BR2)~\cite{Bassi2005} and the symmetric interior penalty method
(SIPG)~\cite{hartmann:2007b, Shahbazi2009multigrid}.

While the stability conditions of Theorem \ref{thm:stability-condition} depend
on $\Siggk{2}$ and $\Siggn{2}$, there remains considerable flexibility in the
values adopted for these matrices, provided they satisfy
\eqref{eqn:assumption_on_sigma2} and the conditions in Theorem
\ref{thm:adjoint_conditions}.  Additionally, although a positive semi-definite
$\Siggk{4}$ may influence the accuracy and continuity of solutions, it is not
necessary nor is it sufficient to guarantee coercivity of the bilinear
form.  Accordingly, a simple and effective choice for the penalty matrices is
\begin{equation*}
\begin{aligned}
\Siggk{3} &= -\Siggk{2} = \frac{1}{2}\Bg, \\
\Siggk{4} &= \Siggn{4} = 0,
\end{aligned}
\end{equation*}
which are the values used for the remainder of the paper.  Note that other
choices are possible that lead to asymmetric or one-sided schemes, such as the
local discontinuous Galerkin scheme~\cite{Cockburn1998_LDG}, but these are not
considered in this paper.

We now investigate two specific expressions for $\Siggam{1}$ and
$\Sig_{\gamma}^{\fnc{D}}$ and show how these are related to BR2 and SIPG.

\subsection{The modified scheme of Bassi and Rebay (BR2)}\label{sec:BR2}

Based on the stability analysis in Section~\ref{sec:energy_analysis},
specifically Theorem~\ref{thm:stability-condition}, a straightforward choice for
the SAT penalties is
\begin{align}
  \Siggam{1} &= \frac{1}{4} \Bg \left[ \Cgk \left(\alpha_{\gamma\kappa}\Lam_{\kappa}^{*}\right)^{-1}\Cgk^T 
  + \Cgn \left(\alpha_{\gamma\nu} \Lam_{\nu}^{*}\right)^{-1}\Cgn^T \right] \Bg, \label{eq:SAT-BR2} \\
  \Sig_{\gamma}^{\fnc{D}} &= \Bg \Cgk \left(\alpha_{\gamma\kappa}\Lam_{\kappa}^{*}\right)^{-1}\Cgk^T \Bg \label{eq:SAT-BR2-boundary}.
\end{align}

We now show that the above penalty matrices generalize the modified scheme of
Bassi and Rebay~\cite{Bassi2005} to multidimensional SBP discretizations.  For
ease of exposition, we will consider the scalar constant-coefficient diffusion
case, that is
\begin{equation*}
  \Lambda = \begin{bmatrix} \lambda_{xx} & \lambda_{xy} \\ \lambda_{yx} & \lambda_{yy}
\end{bmatrix} = \lambda \begin{bmatrix} 1 & \\ & 1 \end{bmatrix}.
\end{equation*}
A similar analysis of problems with a spatially varying tensor $\Lambda$ gives the same conclusion. 
In addition, we will only focus on the interface penalty of BR2, since the
relationship to $\Sig_{\gamma}^{\fnc{D}}$ is similar.

The penalties in the BR2 method that correspond with the matrix $\Siggam{1}$ are of
the form
\begin{equation}\label{eq:BR2}
  C_{\text{BR2}} \int_{\partial \Omega_{\kappa} \cap \Gamma^\fnc{I}}
  \fnc{V}_{\kappa} \frac{1}{2} \left[ n_x \left(\fnc{L}_{x,\kappa}^{\gamma} +
    \fnc{L}_{x,\nu}^{\gamma}\right) + n_y \left(\fnc{L}_{y,\kappa}^{\gamma} +
    \fnc{L}_{y,\nu}^{\gamma}\right) \right] \, \diff \Gamma,
\end{equation}
where $C_{\text{BR2}}$ is a positive constant, $\fnc{U}_{\kappa}$ and
$\fnc{U}_{\nu}$ denote the finite-dimensional solution on the elements
$\Omega_\kappa$ and $\Omega_\nu$, respectively, and $\fnc{V}_{\kappa}$ denotes the test
function on $\Omega_\kappa$.  We have also introduced the \emph{scalar} lifting
operators $\fnc{L}_{x,\kappa}^{\gamma}$ and $\fnc{L}_{y,\kappa}^{\gamma}$, which
are defined by the variational statements
\begin{align*}
  \int_{\Omega_{\kappa}} \fnc{V}_{\kappa} \fnc{L}_{x,\kappa}^{\gamma} \, \diff\Omega
  &= \frac{1}{2} \int_{\gamma} \fnc{V}_{\kappa} \lambda (\fnc{U}_{\kappa} - \fnc{U}_{\nu}) n_x
  \, \diff \Gamma,
  \qquad\forall \fnc{V}_{\kappa} \in \poly{p}(\Omega_\kappa), \\
  \text{and}\qquad
  \int_{\Omega_{\kappa}} \fnc{V}_{\kappa} \fnc{L}_{y,\kappa}^{\gamma} \, \diff\Omega
  &= \frac{1}{2} \int_{\gamma} \fnc{V}_{\kappa} \lambda (\fnc{U}_{\kappa} - \fnc{U}_{\nu}) n_y
  \, \diff \Gamma,
  \qquad\forall \fnc{V}_{\kappa} \in \poly{p}(\Omega_\kappa).
\end{align*}
The multidimensional SBP discretization of these two variational 
statements is
\begin{align*}
  \Vkp^T \Hk \Lxk &= \frac{\lambda}{2} \Vkp^T \Rgk^T \Bg \Nxg (\Rgk \Ukp - \Rgn \Unu), 
  \qquad\forall\; \Vkp \in \mathbb{R}^{n_{\kappa}}, \\
\text{and}\qquad
  \Vkp^T \Hk \Lyk &= \frac{\lambda}{2} \Vkp^T \Rgk^T \Bg \Nyg (\Rgk \Ukp - \Rgn \Unu),
  \qquad\forall\; \Vkp \in \mathbb{R}^{n_{\kappa}},
\end{align*}
where $\Lxk \in \mathbb{R}^{n_{\kappa}}$ and $\Lyk \in \mathbb{R}^{n_{\kappa}}$
are the discrete lifting operators.  Choosing $\Vkp$ appropriately, we obtain
the explicit expressions
\begin{align*}
  \Lxk &= \frac{\lambda}{2} \Hk^{-1} \Rgk^T \Bg \Nxg (\Rgk \Ukp - \Rgn \Unu), \\
  \text{and}\qquad
  \Lyk &= \frac{\lambda}{2} \Hk^{-1} \Rgk^T \Bg \Nyg (\Rgk \Ukp - \Rgn \Unu).
\end{align*}

Next, we turn to the SBP discretization of the BR2 penalty~\eqref{eq:BR2}.
Using the above expressions for $\Lxk$ and $\Lyk$, and the analogous ones for
$\Lxn$ and $\Lyn$, we obtain the discretization
\begin{align*}
  &\frac{C_{\text{BR2}}}{2} \Vkp^T \Rgk^T \Bg \left[ 
    \Nxg (\Rgk\Lxk + \Rgn\Lxn) + \Nyg (\Rgk\Lyk + \Rgn\Lyn) \right] \\
  &= \frac{C_{\text{BR2}}}{4} \Vkp^T \Rgk^T \Bg \lambda \left[
    (\Nxg \Rgk \Hk^{-1} \Rgk^T \Bg \Nxg + \Nyg \Rgk \Hk^{-1} \Rgk^T \Bg \Nyg) 
    \right. \\
  &\phantom{=}\left. + 
    (\Nxg \Rgn \H_{\nu}^{-1} \Rgn^T \Bg \Nxg + \Nyg \Rgn \H_{\nu}^{-1} \Rgn^T \Bg \Nyg)
    \right] (\Rgk \Ukp - \Rgn \Unu)  \\
  &= \Vkp^T \Rgk^T \Sig_{\text{BR2}} (\Rgk \Ukp - \Rgn \Unu),
\end{align*}
where 
\begin{align*}
  \Sig_{\text{BR2}} & =\frac{C_{\text{BR2}}}{4} 
  \Bg \lambda \left[
    (\Nxg \Rgk \Hk^{-1} \Rgk^T \Bg \Nxg + \Nyg \Rgk \Hk^{-1} \Rgk^T \Bg \Nyg) \right. \\
    &\phantom{= \frac{C_{\text{BR2}}}{4}\Bg}
    \left. + 
    (\Nxg \Rgn \H_{\nu}^{-1} \Rgn^T \Bg \Nxg + \Nyg \Rgn \H_{\nu}^{-1} \Rgn^T \Bg \Nyg)
    \right] \\
  &= \frac{C_{\text{BR2}}}{4} \Bg \left\{ \begin{bmatrix} \Nxg \Rgk & \Nyg \Rgk \end{bmatrix}
  \begin{bmatrix} \lambda \Hk^{-1} & \\ & \lambda \Hk^{-1} \end{bmatrix}
  \begin{bmatrix} \Rgk^T \Nxg \\ \Rgk^T \Nyg \end{bmatrix} \right. \\
  &\phantom{= C_{\text{BR2}}}
  \left. + 
  \begin{bmatrix} \Nxg \Rgn & \Nyg \Rgn \end{bmatrix}
  \begin{bmatrix} \lambda \H_{\nu}^{-1} & \\ & \lambda \H_{\nu}^{-1} \end{bmatrix}
  \begin{bmatrix} \Rgn^T \Nxg \\ \Rgn^T \Nyg \end{bmatrix} \right\} \Bg \\
  &= \frac{C_{\text{BR2}}}{4} \Bg \left[ \Cgk (\Lam_{\kappa}^{*})^{-1} \Cgk^T 
  + \Cgn (\Lam_{\nu}^{*})^{-1} \Cgn^T \right] \Bg.
\end{align*}
In the above derivation, we reversed the direction of $\Nxg$ and $\Nyg$ for
element $\Omega_\nu$, and we used $\Bg \Nxg = \Nxg \Bg$ and $\Bg \Nyg = \Nyg \Bg$.

From the above expression for $\Sig_{\text{BR2}}$, we see that
\eqref{eq:SAT-BR2} is indeed the SBP generalization of the BR2 penalty
\eqref{eq:BR2} with $C_{\text{BR2}} = \alpha_{\gamma\kappa}^{-1}$.  In the
results section, we will refer to this scheme as SAT-BR2.

\begin{remark}
  To the best of our knowledge, this is the first time the SBP generalization of
  the BR2 scheme has been presented.  This is significant, because it provides a
  means of implementing the popular BR2 scheme with multidimensional SBP
  operators that do not have underlying basis functions.
\end{remark}

\subsection{The symmetric interior penalty method (SIPG)}\label{sec:SIPG}

A disadvantage of the SAT-BR2 penalties is that their $\Siggam{1}$ and
$\Sig_{\gamma}^{\fnc{D}}$ matrices can be computationally expensive to evaluate.  This
is not an issue for linear problems --- since these matrices can be pre-computed
and stored if sufficient memory is available --- but it can be an issue in
nonlinear problems when the diffusion coefficient(s) depend on the state.

In contrast to dense penalty matrices, the symmetric interior penalty method
(SIPG)~\cite {Hartmann2005, hartmann:2007b, Shahbazi2005explicit, Shahbazi2009multigrid} uses diagonal (or block diagonal) $\Siggam{1}$ and
$\Sig_\gamma^{\fnc{D}}$ with a single parameter that is chosen to be sufficiently
large to ensure stability.  In this section, we demonstrate how the
multidimensional SBP-SAT generalization of SIPG can be derived from the conditions in
Theorem~\ref{thm:stability-condition}.  First, we need the following lemma.

\begin{lemma}\label{lem:delta-bound}
  Let $(\lambda_{\max})_{\kappa}$ be the largest eigenvalue of
  $\left[\begin{smallmatrix} \Lamxx & \Lamxy \\ \Lamyx &
      \Lamyy \end{smallmatrix}\right]_{\kappa}$ and let $\| \mat{A} \|_{2} =
  \sqrt{\rho(\mat{A}\mat{A}^T)}$ denote the matrix 2-norm.  Then
  \begin{equation}
    \Bg \Cgk \left(\alpha_{\gamma\kappa}\Lam_{\kappa}^{*}\right)^{-1}\Cgk^T \Bg
    \preceq 
    \frac{(\lambda_{\max})_\kappa 
      \| \Bg^{\frac{1}{2}} \Rgk \Hk^{-\frac{1}{2}} \|_{2}^{2}}%
         {\alpha_{\gamma\kappa}} \Bg.
    \label{eq:delta-bound}
  \end{equation}
\end{lemma}

\begin{proof}
  We recall a few facts that will be useful.  The matrices $\Bg$, $\Nxg$, and
  $\Nyg$ are diagonal; therefore, they commute with one another.  Furthermore,
  the diagonal matrix $\Bg$ holds positive cubature weights on its diagonal, so
  it can be factored as $\Bg = \Bg^{\frac{1}{2}} \Bg^{\frac{1}{2}}$.  For
  similar reasons we can write $\Hk = \Hk^{\frac{1}{2}}\Hk^{\frac{1}{2}}$.
  Finally, $\Nxg\Nxg^T + \Nyg\Nyg^T = \mat{I}$, since $\Nxg$ and $\Nyg$ hold the
  $x$ and $y$ components of the unit normal along $\gamma$.

  Now, let $\Ugam \in \mathbb{R}^{n_{\gamma}}$ be an arbitrary solution on the
  nodes of the face $\gamma$.  Then products with $\Ugam$ and the matrix on the
  left of \eqref{eq:delta-bound} can be bounded as follows:
  \begin{align*}
    &\Ugam^T \Bg \Cgk \left(\alpha_{\gamma\kappa}\Lam_{\kappa}^{*}\right)^{-1}\Cgk^T \Bg \Ugam \\
    &= 
    \frac{1}{\alpha_{\gamma\kappa}} 
    \Ugam^T \Bg \begin{bmatrix} \Nxg\Rgk & \Nyg\Rgk \end{bmatrix}
    \begin{bmatrix} \Hk^{-1} & \\ & \Hk^{-1} \end{bmatrix}
    \begin{bmatrix} \Lamxx & \Lamxy \\ \Lamyx & \Lamyy \end{bmatrix}
    \begin{bmatrix} \Rgk^T \Nxg^T \\ \Rgk^T \Nyg^T \end{bmatrix} \Bg \Ugam \\
    &\leq \frac{(\lambda_{\max})_{\kappa}}{\alpha_{\gamma\kappa}} 
    \Ugam^T\Bg^{\frac{1}{2}}
    \begin{bmatrix} \Nxg & \Nyg \end{bmatrix}
    \left\{ \begin{bmatrix} 1 & \\ & 1 \end{bmatrix} \otimes \left(
    \Bg^{\frac{1}{2}} \Rgk \Hk^{-1} \Rgk^T \Bg^{\frac{1}{2}} \right) \right\}
    \begin{bmatrix} \Nxg^T \\ \Nyg^T \end{bmatrix} \Bg^{\frac{1}{2}} \Ugam \\
    &\leq \frac{(\lambda_{\max})_{\kappa} \| \Bg^{\frac{1}{2}} \Rgk \Hk^{-\frac{1}{2}} \|_{2}^{2}}{\alpha_{\gamma\kappa}} 
    \Ugam^T \Bg^{\frac{1}{2}} 
    \begin{bmatrix} \Nxg & \Nyg \end{bmatrix} 
    \begin{bmatrix} \Nxg^T \\ \Nyg^T \end{bmatrix} \Bg^{\frac{1}{2}} \Ugam \\
    &\leq \frac{(\lambda_{\max})_{\kappa} 
      \| \Bg^{\frac{1}{2}} \Rgk \Hk^{-\frac{1}{2}} \|_{2}^{2}}{\alpha_{\gamma\kappa}}
    \Ugam^T \Bg \Ugam.
  \end{align*}
  The desired result follows from the above inequality, since $\bm{u}_{\gamma}$
  is arbitrary. \qed
\end{proof}

We can now state the SAT-SIPG penalties that lead to energy stability.

\begin{theorem}
  The discretization~\eqref{eq:parabolic_SBP} is energy stable if 
  \begin{equation} \label{eq:SIPG}
    \Siggam{1} = \delta_{\gamma}^{(1)} \Bg,
    \qquad\text{and}\qquad
    \Sig_{\gamma}^{\fnc{D}} = \delta_{\gamma}^{\fnc{D}} \Bg,
  \end{equation}
  where 
  \begin{align*}
    \delta_{\gamma}^{(1)} &= \frac{(\lambda_{\max})_\kappa 
      \| \Bg^{\frac{1}{2}} \Rgk \Hk^{-\frac{1}{2}} \|_{2}^{2}}%
          {4 \alpha_{\gamma\kappa}} 
          + \frac{(\lambda_{\max})_\nu 
            \| \Bg^{\frac{1}{2}} \Rgn \H_{\nu}^{-\frac{1}{2}} \|_{2}^{2}}%
          {4 \alpha_{\gamma\nu}}, \\ 
          \delta_{\gamma}^\fnc{D} &= \frac{(\lambda_{\max})_\kappa 
      \| \Bg^{\frac{1}{2}} \Rgk \Hk^{-\frac{1}{2}} \|_{2}^{2}}%
          {\alpha_{\gamma\kappa}} .
  \end{align*}
\end{theorem}

\begin{proof}
  The proof follows from Lemma~\ref{lem:delta-bound}, the conditions in
  Theorem~\ref{thm:stability-condition}, and the aforementioned choice
  $\Siggk{3} = -\Siggk{2} = 1/2\Bg$.\qed
\end{proof}

\begin{remark}
    The SAT-SIPG penalties require that we compute $(\lambda_{\max})_{\kappa}$ on
    each element.  For nonlinear problems, we recommend replacing this value with
    an estimate for the upper bound of the spectral radius of the tensor $\Lambda$
    over all nodes of $\kappa$; otherwise the computational advantage of SAT-SIPG
    over SAT-BR2 will be compromised. 
\end{remark}

\begin{remark}
    To the best of our knowledge, this is the first time the SIPG penalty has been related to BR2 using straightforward matrix analysis. The approximations that lead to SIPG produce a more conservative bound, on the one hand, but a cheaper penalty, on the other hand. 
    
\end{remark}

The SIPG penalty parameters $\delta_{\gamma}^{(1)}$ and $\delta_{\gamma}^{\fnc{D}}$
are similar to those given by Shahbazi~\cite{Shahbazi2005explicit}.
Indeed, we
have verified that they are identical for degree $p$ operators on simplex elements with constant-coefficient scalar diffusion, provided
\begin{enumerate}
\item the SBP matrices $\Hk$ and $\Bg$ and their corresponding nodes define cubature rules
that are exact for polynomials of degree $2p$, and; 
\item the number of SBP nodes is equal to the number of basis functions for $\poly{p}$.
\end{enumerate}
When these two conditions are satisfied the SBP cubatures reproduce the $L^2$
norm on the volume and face exactly, so the inverse trace inequalities of
Warburton and Hesthaven apply~\cite{Warburton20032765}.  However, in general, $\Hk$ is only
exact for polynomials of degree $2p-1$ and there are more SBP nodes than basis
functions in $\poly{p}$, so the penalties given here differ
from~\cite{Shahbazi2005explicit}.  Furthermore, the penalties $\delta_\gamma^{(1)}$ and $\delta_\gamma^{\mathcal{D}}$ are more general than those provided in~\cite{Shahbazi2005explicit}, because they are applicable to spatially varying tensor diffusion and elements other than simplices.

In practice, we define SBP operators on a reference element and employ a
coordinate transformation for each element in the physical domain.  Therefore, some remarks
are warranted regarding the implementation of the SAT-SIPG penalties when
coordinate transformations are used.  Let $\bm{x}(\bm{\xi})$ be an affine and
bijective coordinate transformation from reference space, $\bm{\xi} =
[\xi,\eta]^T \in \Omega_{\xi}$, to physical space.  When such a coordinate
mapping is used with SAT-SIPG penalties, $(\lambda_{\max})_{\kappa}$ corresponds
to the largest eigenvalue of
\begin{equation*}
\begin{bmatrix} \mat{J} \Lamxx & \mat{J} \Lamxy \\ \mat{J} \Lamyx &
\mat{J} \Lamyy \end{bmatrix}_{\kappa}, 
\end{equation*}
where $\mat{J}$ is a diagonal matrix holding the determinant of the mapping
Jacobian at each node of $\kappa$.  In addition, the penalties matrices
$\Siggam{1}$ and $\Sig_{\gamma}^{\fnc{D}}$ in \eqref{eq:SIPG} must be multiplied by
the squared norm of the scaled contravariant basis vectors at the face nodes,
\ie, the diagonal matrix whose $j$th entry is
\begin{equation*}
\left\| \left[ \fnc{J} (n_\xi \nabla \xi + n_\eta \nabla \eta)\right]_{j} \right\|^{2}
,\qquad \forall j = 1,2,\ldots,n_{\gamma},
\end{equation*}
where $\fnc{J} = \det(\partial\bm{x}/\partial \bm{\xi})$ is the determinant of
the mapping Jacobian, $\nabla \xi$ and $\nabla \eta$ are the contravariant basis
vectors, and $n_\xi$ and $n_\eta$ are the components of the unit normal on face
$\gamma$ in reference space.  Note that the squared norm $\| \Bg^{\frac{1}{2}}
\Rgk \Hk^{-\frac{1}{2}} \|_{2}^{2}$ can be pre-computed in reference space.

\section{Numerical experiments}\label{sec:results}

This section presents some numerical experiments to verify the theory developed
in Sections~\ref{sec:adjoint}, \ref{sec:energy_analysis}, and
\ref{sec:generalize}.  For the verifications, we consider two families of SBP
operators developed for simplex elements.
\begin{description}
\item[SBP-$\Omega$:] These operators have strictly internal nodes, and the
  number of nodes is equal to the number of basis functions in
  $\poly{p}(\Omega_{\xi})$; therefore, the SBP-SAT discretizations based on these
  operators are equivalent to collocation discontinuous-Galerkin finite-element
  methods.  For the degree $p=1$ and $p=2$ operators, the SBP norm is a $2p$
  degree cubature, while for $p=3$ and $p=4$, the norm is a degree $2p-1$
  cubature.  Thus, for constant coefficient-diffusion, the SIPG penalty is
  identical to Shahbazi's for $p=1$ and $p=2$, while it is different for $p=3$
  and $p=4$ (see the discussion in Section \ref{sec:SIPG}).  The
  SBP-$\Omega$ operators were first presented in \cite{Fernandez2016}; see
  also~\cite{Hicken2016}.
\item[SBP-$\Gamma$:] These operators were designed to have $p+1$ nodes on each
  face; consequently, the interpolation operator $\Rgk$ uses only those nodes
  that lie on face $\gamma$.  With the exception of $p=1$, the SBP-$\Gamma$
  operators have more SBP nodes than basis functions in
  $\poly{p}(\Omega_{\xi})$.  For this reason, there are no (known) basis
  functions associated with these operators for $p > 1$; they are
  finite-difference operators but not finite-element operators.  The
  SBP-$\Gamma$ operators were presented in~\cite{multiSBP}.
\end{description}
Table~\ref{tab:operators} summarizes the SBP-$\Omega$ and SBP-$\Gamma$ operators
considered in this work.  For further details on the construction of these
operators, please see \cite{multiSBP} and \cite{Fernandez2016}.

\begin{table}[t]
  \begin{center}
    \caption[]{The SBP-$\Omega$ and SBP-$\Gamma$ operators for the triangle. The
      open circles denote the locations of the SBP nodes, while the black
      squares denote the locations of the face cubature points (for a given
      degree $p$ SBP operator, the face cubatures are the same for both
      families). \label{tab:operators}} \setlength{\tabcolsep}{0pt}
    \begin{tabular}{p{0.15\textwidth}p{0.2\textwidth}p{0.2\textwidth}p{0.2\textwidth}p{0.2\textwidth}}
      & \multicolumn{4}{c}{\textbf{degree}} \\\cline{2-5}
      \textbf{family} & 
      \multicolumn{1}{c}{$p=1$} & 
      \multicolumn{1}{c}{$p=2$} & 
      \multicolumn{1}{c}{$p=3$} & 
      \multicolumn{1}{c}{$p=4$}  \rule{0ex}{3ex} \\\hline
      \vspace*{-0.12\textwidth}\textbf{SBP}-$\bm{\Omega}$ &
      \parbox[b]{0.2\textwidth}{%
        \begin{center}%
          \includegraphics[width=0.16\textwidth]{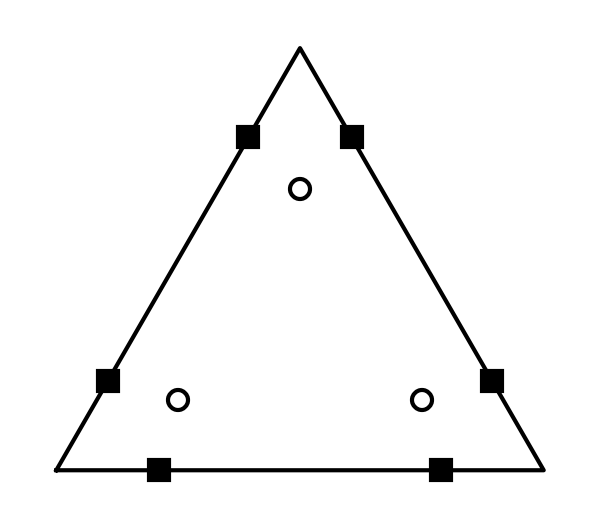}\\
          3 nodes\end{center}} &
      \parbox[b]{0.2\textwidth}{%
        \begin{center}%
          \includegraphics[width=0.16\textwidth]{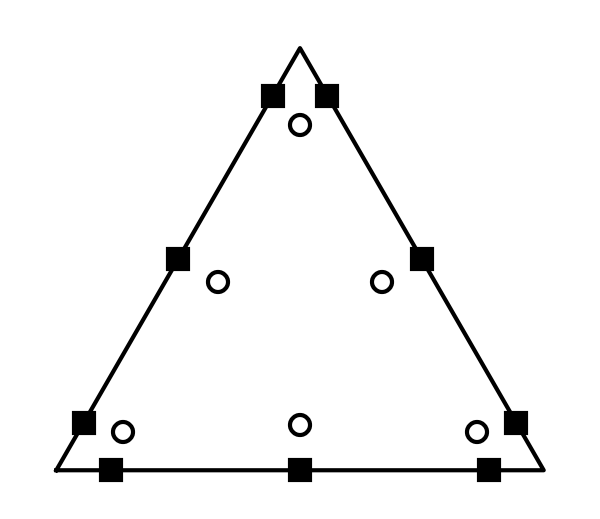}\\
          6 nodes\end{center}} &
      \parbox[b]{0.2\textwidth}{%
        \begin{center}%
          \includegraphics[width=0.16\textwidth]{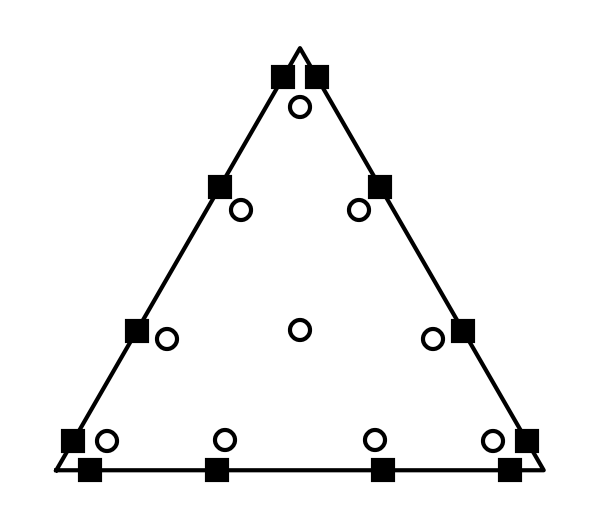}\\
          10 nodes\end{center}} &
      \parbox[b]{0.2\textwidth}{%
        \begin{center}%
          \includegraphics[width=0.16\textwidth]{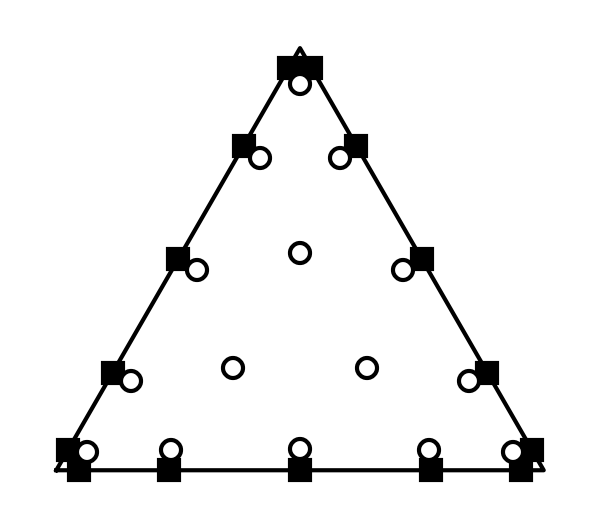}\\
          15 nodes\end{center}} \\\hline
      \vspace*{-0.12\textwidth}\textbf{SBP}-$\bm{\Gamma}$ & 
      \parbox[b]{0.2\textwidth}{%
        \begin{center}%
          \includegraphics[width=0.16\textwidth]{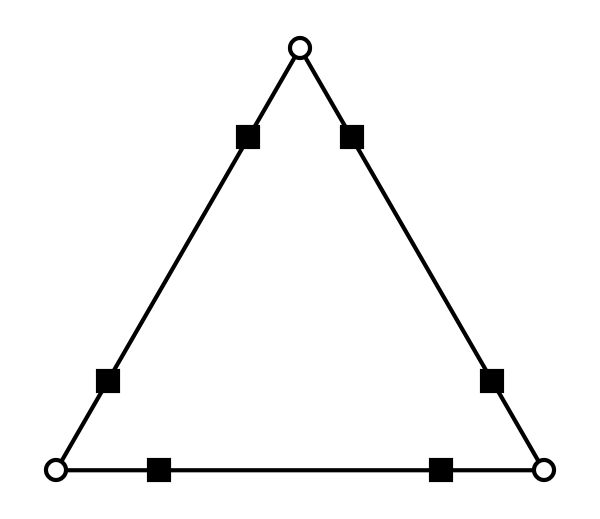}\\
          3 nodes\end{center}} &
      \parbox[b]{0.2\textwidth}{%
        \begin{center}%
          \includegraphics[width=0.16\textwidth]{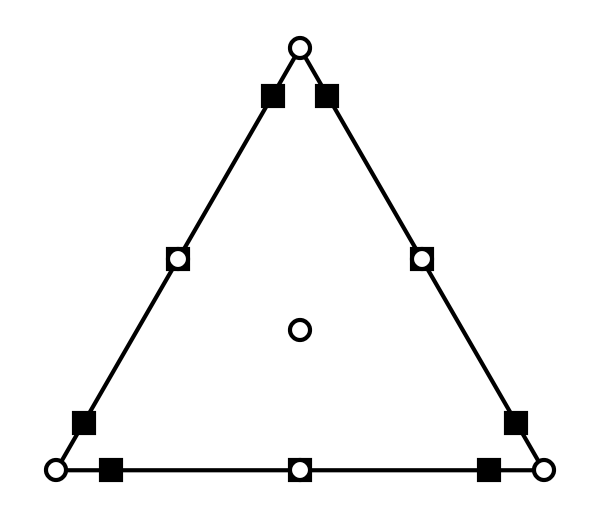}\\
          7 nodes\end{center}} &
      \parbox[b]{0.2\textwidth}{%
        \begin{center}%
          \includegraphics[width=0.16\textwidth]{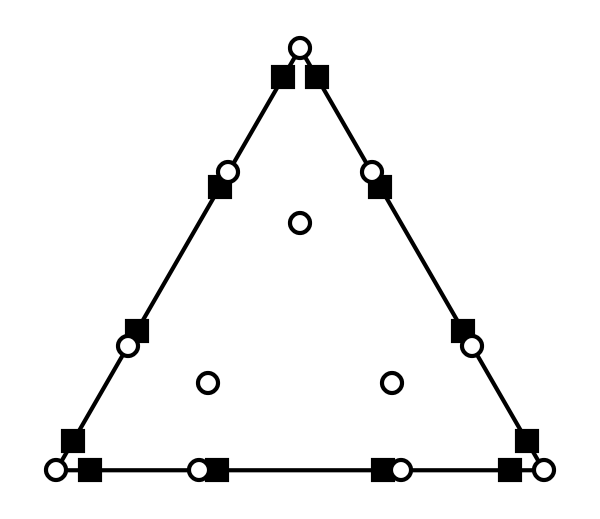}\\
          12 nodes\end{center}} &
      \parbox[b]{0.2\textwidth}{%
        \begin{center}%
          \includegraphics[width=0.16\textwidth]{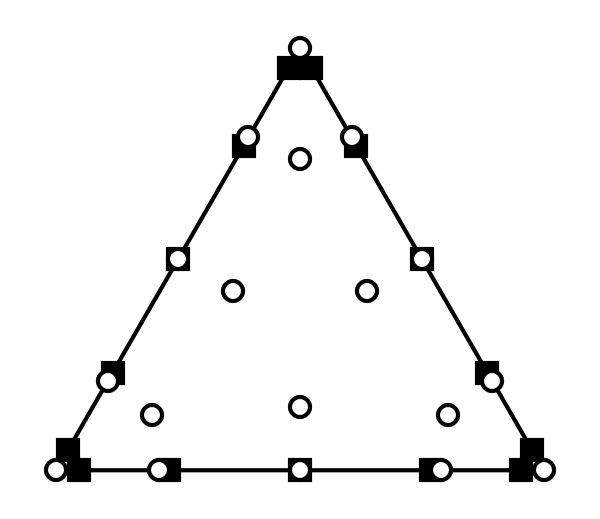}\\
          18 nodes\end{center}} \\\hline
    \end{tabular}
  \end{center}
\end{table}

The SAT-SIPG and SAT-BR2 generalizations in Section \ref{sec:generalize} are
implemented with face-weight parameters, $\alpha_{\gamma\kappa}$, computed using
the face area as follows.
\begin{equation*}
\alpha_{\gamma\kappa} = 
\begin{cases}
\cfrac{\fnc{A}(\gamma)}{\fnc{A}(\Gamma^{\fnc{I}}_{\kappa}) + 2\fnc{A}(\Gamma^{\fnc{D}}_{\kappa})} ,\quad\quad\gamma \in\Gamma^{\fnc{I}}, \\
\cfrac{2\fnc{A}(\gamma)}{\fnc{A}(\Gamma^{\fnc{I}}_{\kappa}) + 2\fnc{A}(\Gamma^{\fnc{D}}_{\kappa})},\quad\quad\gamma \in\Gamma,
\end{cases}
\end{equation*}
where the function $\fnc{A(\gamma)}$ computes the size of face $\gamma$, \ie,
length in 2D and area in 3D. The condition $\sum_{\gamma \subset
  \Gamma_{\kappa}} \alpha_{\gamma\kappa} = 1$ required in Lemma
\ref{lem:energy_start} is clearly satisfied by the above definition.

As noted above, the interpolation operator $\Rgk$ for the SBP-$\Gamma$
discretizations depends only on the nodes on the boundary of $\gamma$, so the
cost of applying $\Rgk$ or $\Rgk^T$ is $\text{O}(p)$ in two-dimensions.  In
three dimensions, the SBP-$\Gamma$ operators have $(p+1)(p+2)/2$ nodes on each
face and the cost of applying the interpolation operator and its transpose is
$\text{O}(p^2)$.  A consequence of this sparsity structure of $\Rgk$, as well as
the norm $\Hk$ being diagonal, is that the SAT-BR2 penalty matrix
$\Sig_{\text{BR2}}$ has an asymptotic cost of $\text{O}(p)$ and $\text{O}(p^2)$
in two and three dimensions, respectively, when SBP-$\Gamma$ operators are used.

In contrast, the SBP-$\Omega$ operators require dense extrapolation operators,
so the cost of applying $\Rgk$ and $\Rgk^T$ --- and, consequently, the cost of
$\Sig_{\text{BR2}}$ --- scales as $\text{O}(p^2)$ in two dimensions and
$\text{O}(p^3)$ in three dimensions.

\subsection{Accuracy study}
 
The first experiment is intended to verify primal and adjoint consistency by
examining the convergence rates of a discrete solution and an associated
functional.  To this end, we use a manufactured solution on the unit square
$\Omega = [0,1]^2$.  Specifically, we adopt a second-order polynomial function
for $\Lambda$ and a trigonometric solution for $\fnc{U}$ defined by,
respectively,
\begin{equation} \label{eqn:trigonometric forcing function}
\begin{aligned}
\Lambda &= \begin{bmatrix}
x^2+1 & xy \\ xy & y^2+1
\end{bmatrix}, \\
\text{and}\qquad
\fnc{U}(x,y) &= \sin(2\pi x)\sin(2\pi y).
\end{aligned}
\end{equation}
The source term $\fnc{F}$ is found by substituting $\Lambda$ and $\fnc{U}$ into
\eqref{eq:parabolic}.  Homogeneous Dirichlet boundary conditions are applied
along the boundaries of $\Omega$.  In addition, the time derivative in
\eqref{eq:parabolic} is dropped, \ie, we consider the Poisson PDE, since we are
interested in the accuracy of the spatial discretization.  The functional is
defined as
\begin{equation} \label{eqn:functional-for-test}
\begin{aligned}
\fnc{J} = \int_{\Omega} \fnc{U}\diff \Omega,
\end{aligned}
\end{equation}
which is a special case of \eqref{eq:fun} with $\fnc{V}_\fnc{N} = 0$ and
$\fnc{V}_\fnc{D} = 0$.  To estimate the asymptotic convergence rates, we use a
sequence of uniformly refined meshes consisting of $K = 128$, $512$, $2048$, and
8192 triangular elements. The coarsest mesh is shown in Figure \ref{fig:coarse
  mesh}.  The nominal element size is given by $h \equiv 1/\sqrt{K/2}$.

Figure \ref{fig:p_accuracy} shows the solution error measured in terms of the
$L^2$-norm; the integral in the $L^2$ norm is approximated using the SBP
matrices $\H_{\kappa}$ scaled appropriately by the mapping Jacobian.  We see
that, under mesh refinement, the solution errors are asymptotically
$O(h^{p+1})$, which is in agreement with the design accuracy.

Figure \ref{fig:j_accuracy} plots the errors in the SBP-SAT approximation of the
functional \eqref{eqn:functional-for-test}. A convergence rate of $2p$ is
achieved for both operators, which is also in agreement with the theoretical
order of convergence~\cite{hartmann:2007b, Hicken2011superconvergent}.

For this particular problem, the results show that for lower order
approximations ($p=1$ and $p=2$), the SBP-$\Omega$ operators produce more
accurate solutions, typically by 25\%. However, for the higher order
discretizations ($p=3$ and $p=4$) the solutions based on SBP-$\Gamma$ are more
accurate.  The SBP-$\Gamma$ operators produce more accurate estimates of the
functional for all orders of accuracy; however, we caution against generalizing
any of these results based on this one case.

\ignore{
Show optimal L2 rates for a steady problem.  It would be good to choose at least
one method that is not adjoint consistent to show that the L2 rate is not optimal.
}
\begin{figure}
    \begin{subfigure}{0.32\textwidth}
        \centering
        \includegraphics[width=1.0\linewidth]{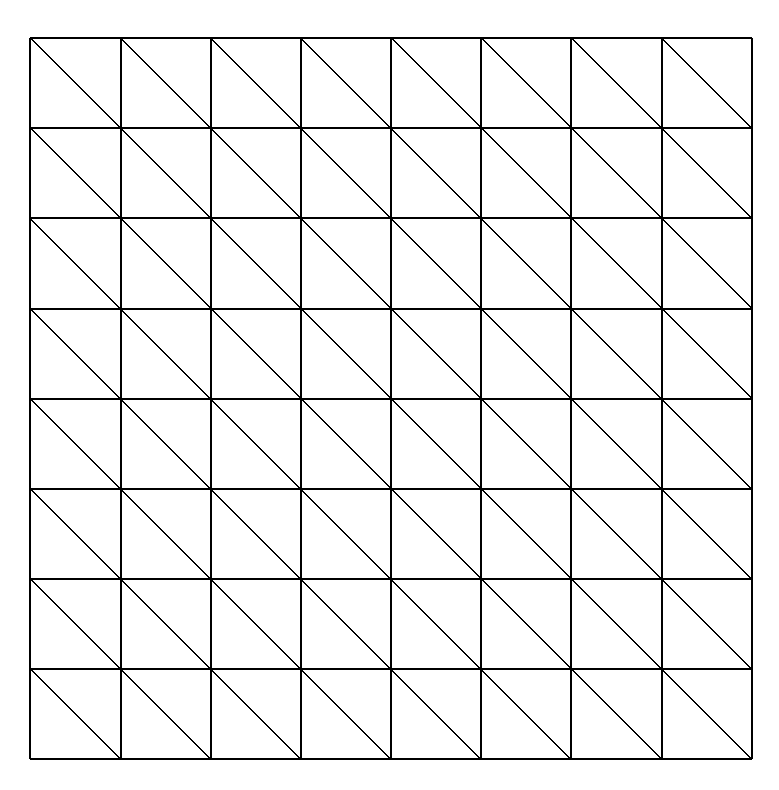}
        \subcaption{Coarsest mesh}
        \label{fig:coarse mesh}
    \end{subfigure}
    \begin{subfigure}{0.32\textwidth}
        \centering
        \includegraphics[width=1.0\linewidth]{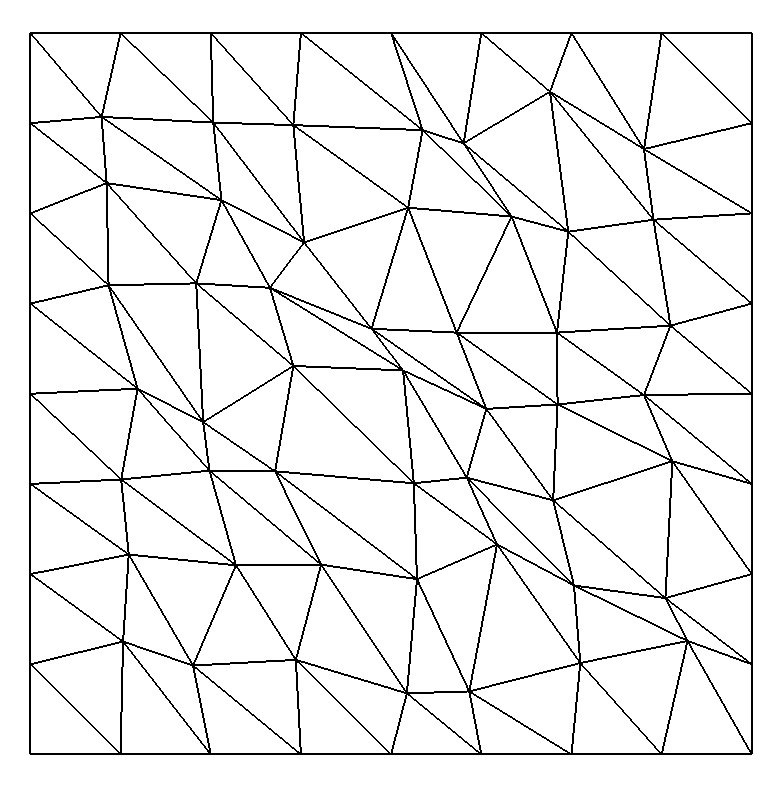}
        \subcaption{Perturbed mesh, $8\times 8$}
        \label{fig:perturbed_mesh_8x8}
    \end{subfigure}
    \begin{subfigure}{0.32\textwidth}
        \centering
        \includegraphics[width=1.0\linewidth]{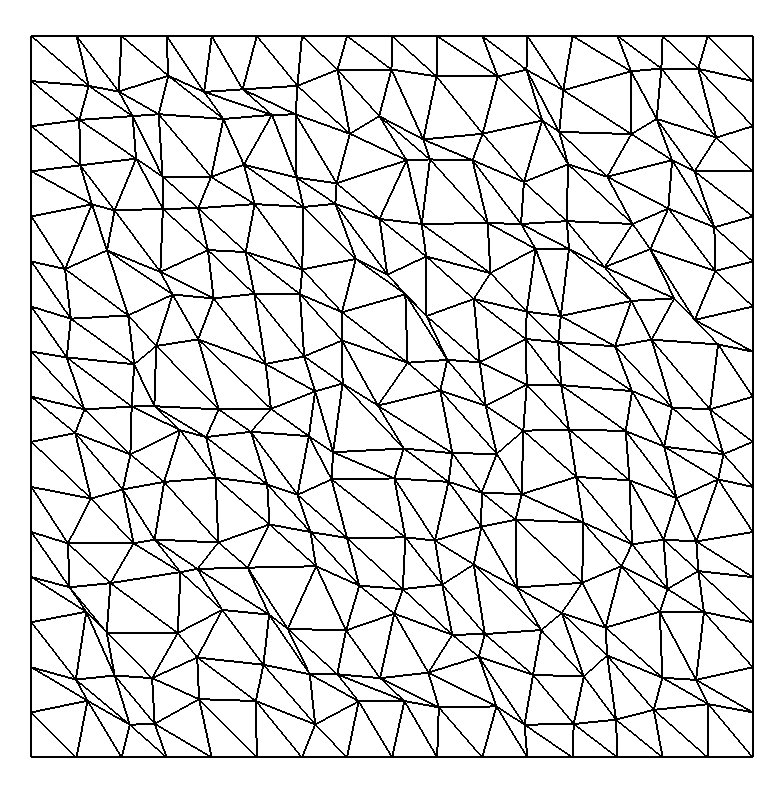}
        \subcaption{Perturbed mesh, $16\times 16$}
        \label{fig:perturbed_mesh}
    \end{subfigure}
    \caption{Different meshes used for test cases} 
\end{figure}

\begin{figure}  
    \centering
    \begin{subfigure}[b]{0.52\linewidth}
        \centering
        \includegraphics[width=1.0\linewidth]{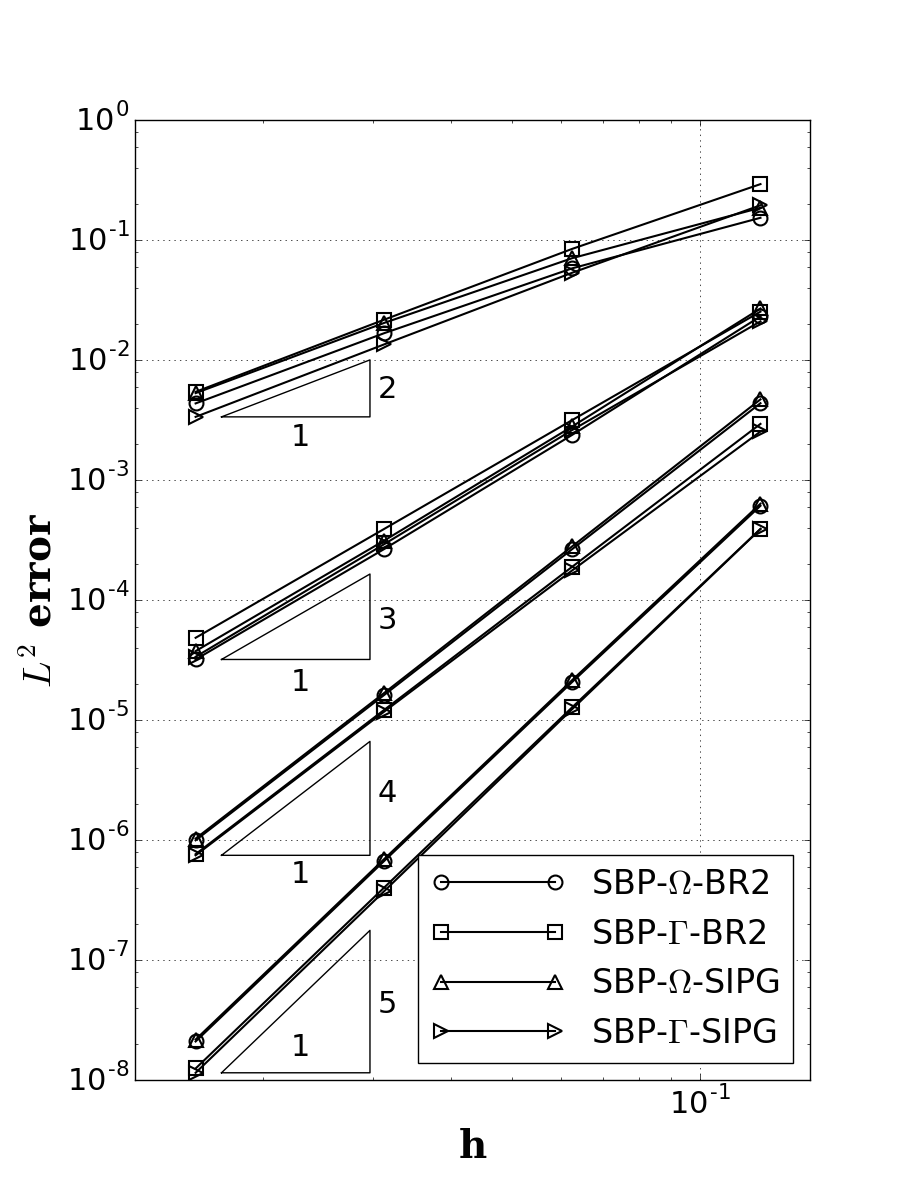}
        \subcaption{Convergence rate of solution}
        \label{fig:p_accuracy}
    \end{subfigure}%
    \begin{subfigure}[b]{0.52\linewidth}
        \centering
        \includegraphics[width=1.0\linewidth]{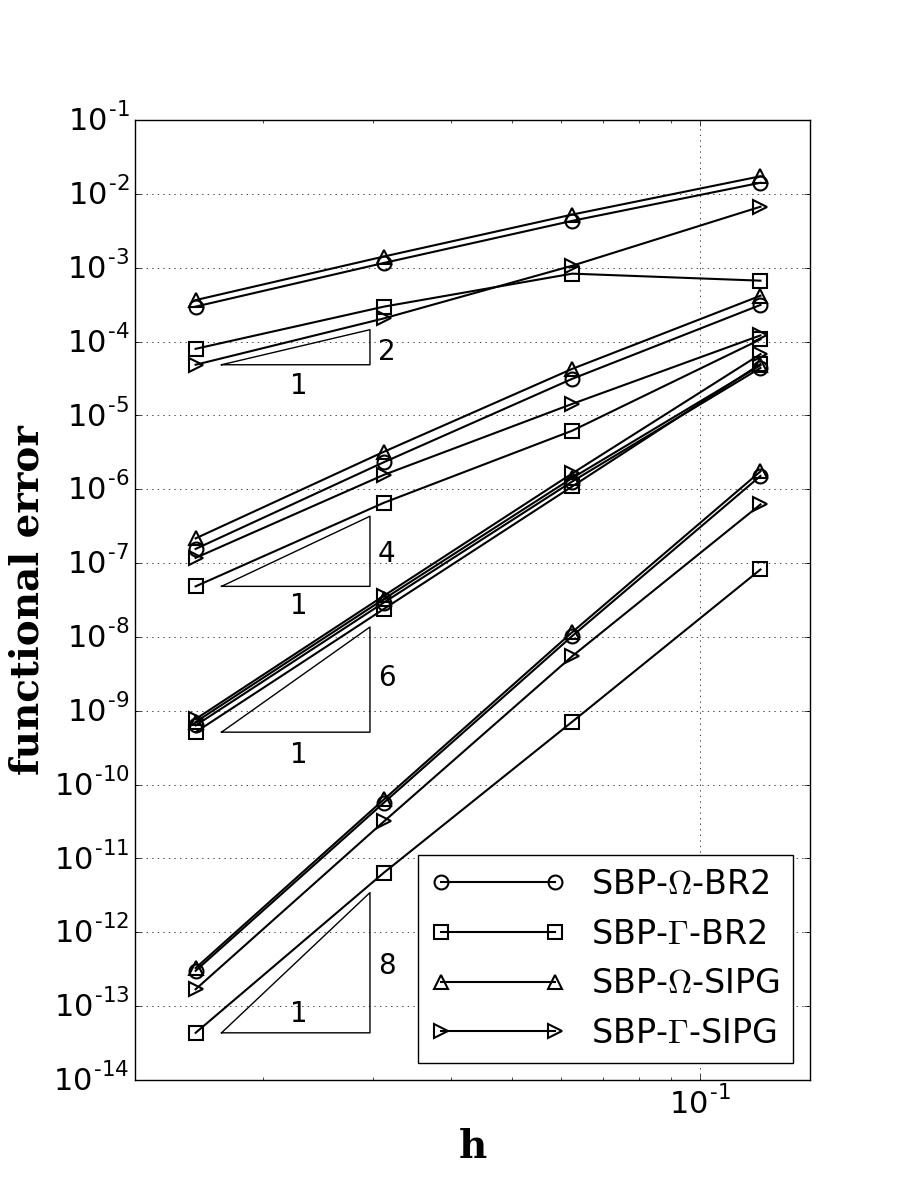}
        \subcaption{Convergence rate of functional}
        \label{fig:j_accuracy}
    \end{subfigure}
    \caption{Convergence rate study}
\end{figure}

%

\subsection{Comparison between SAT-BR2 and SAT-SIPG}

Interior penalties can have a significant influence on the conditioning of the
discretized systems; therefore, in this experiment we assess the conditioning
imparted by the SAT-BR2 and SAT-SIPG penalties.  We expect that SAT-BR2 will
produce lower condition numbers, because it provides a tighter stability bound
than SAT-SIPG.  To numerically verify this, we compare the condition numbers of
the linear systems produced by the SAT-BR2 and SAT-SIPG penalties.  We also
compare the corresponding solution errors, because there is often a trade-off
between conditioning and accuracy.

The penalties are mesh dependent, so this experiment is performed on a
structured triangular mesh that is randomly perturbed, as shown in Figure
\ref{fig:perturbed_mesh}.  The mesh is extremely nonsmooth and almost tangled;
indeed, the largest angle in the mesh is 179.90\textdegree.

The $L^2$ errors and condition numbers are listed in Tables
\ref{tb:l2error_16x16} and \ref{tb:cn_16x16}, respectively, for different orders
of approximations and SBP families. Table \ref{tb:l2error_16x16} shows that the
SAT-BR2 and SAT-SIPG schemes produce comparable errors in the $L^2$ norm for
both SBP families.  This is also observable in the results of the mesh
refinement study presented earlier.  While the BR2 variant is somewhat more
accurate for the SBP-$\Omega$ family, no obvious trend is apparent for the
SBP-$\Gamma$ family.

In contrast, Table \ref{tb:cn_16x16} shows that the BR2 variant consistently
produces a smaller condition number for both families, as expected.
Furthermore, at least for this particular experiment, SBP-$\Gamma$ always
produces considerably lower condition numbers than SBP-$\Omega$.  \ignore{
  Although not shown in this experiment, SAT-BR2 can sometimes be an order of
  magnitude better in conditioning when the condition number of $\Lambda$ is
  large.  }

\begin{table}[!ht] 
    \centering
    \caption{$L^2$-norm of error on the perturbed $16\times 16$
      mesh} \label{tb:l2error_16x16}
    \begin{tabular} {p{1cm} p{2cm}  p{2cm} p{2cm} p{2cm}}
        & \multicolumn{2}{c}{\textbf{SBP-$\bm{\Gamma}$}} & \multicolumn{2}{c}{\textbf{SBP-$\bm{\Omega}$}} \\
        \cline{2-5}
       \rule{0ex}{3ex} & \textbf{SAT-BR2}  &  \textbf{SAT-SIPG}    & \textbf{SAT-BR2} & \textbf{SAT-SIPG} \\\hline
        \rule{0ex}{3ex}$\bm{p=1}$ & 1.160e-1 & 8.030e-2 & 7.802e-2 & 9.045e-2 \\
        $\bm{p=2}$ & 5.274e-3 & 4.720e-3 & 4.969e-3 & 5.697e-3 \\
        $\bm{p=3}$ & 4.103e-4 & 4.026e-4 & 5.866e-4 & 6.274e-4 \\
        $\bm{p=4}$ & 3.132e-5 & 3.163e-5 & 5.016e-5 & 5.236e-5 \\
    \end{tabular}
\end{table}

\begin{table}[!ht] 
    \centering
    \caption{Condition number on the perturbed $16\times 16$ mesh} \label{tb:cn_16x16}
    \begin{tabular} {p{1cm} p{2cm}  p{2cm} p{2cm} p{2cm}}
        & \multicolumn{2} {c} {\textbf{SBP-$\bm{\Gamma}$}} & \multicolumn{2} {c} {\textbf{SBP-$\bm{\Omega}$}} \\
        \cline{2-5}
        \rule{0ex}{3ex} & \textbf{SAT-BR2}  &  \textbf{SAT-SIPG}    & \textbf{SAT-BR2} & \textbf{SAT-SIPG} \\\hline
        \rule{0ex}{3ex}$\bm{p=1}$ & 4.734e5 & 4.852e5 & 1.572e6 & 1.705e6 \\
        $\bm{p=2}$ & 2.843e6 & 2.917e6 & 5.514e6 & 6.354e6 \\
        $\bm{p=3}$ & 6.977e6 & 7.032e6 & 1.789e7 & 2.008e7 \\
        $\bm{p=4}$ & 1.337e7 & 1.440e7 & 3.794e7 & 4.201e7 \\
    \end{tabular}
\end{table}

\subsection{Tightness of the stability bound}

Since the stability conditions for both SAT-BR2 and SAT-SIPG are sufficient but
not necessary, a relaxation factor $\alpha \in (0,1]$ acting on $\Sigma^{(1)}$
  may still yield a stable bilinear form. To some degree, such a relaxation
  factor can serve as a measure of the tightness of the stability
  conditions. For example, overly conservative SAT penalties will allow for a
  relaxation factor that is much smaller than $1$; on the other hand, a
  necessary and sufficient stability condition would only permit $\alpha \geq
  1$.

To study the tightness of the stability conditions, we consider the effect of
scaling the penalties on the largest eigenvalue of the linear system, \ie the
eigenvalue with the smallest magnitude.  In particular, we form the
discretization for the problem defined by \eqref{eqn:trigonometric forcing
  function} but scale $\Siggam{1}$ by $\alpha$, then evaluate the linear
system's largest eigenvalue and plot it versus the relaxation factor.  Due to
the expense of computing the smallest magnitude eigenvalue of large matrices,
the matrices here are evaluated on the coarser $8\times8$ randomly perturbed
mesh shown in Figure \ref{fig:perturbed_mesh_8x8}.

The eigenvalues of smallest magnitude are plotted in Figure
\ref{fig:relaxation}. As can be seen, the allowable relaxation factors are less
than one for both penalties and all SBP operators considered, which verifies the
conditions in Theorem \ref{thm:stability-condition}.  Furthermore, the smallest
allowable relaxation factor is between $0.4$ and $0.6$ for SBP-$\Gamma$ and
between $0.25$ and $0.4$ for SBP-$\Omega$, which suggests that the bound is
relatively tight in the sense that it is not orders of magnitude larger than
necessary.

Finally, the comparison between SAT-BR2 and SAT-SIPG is in agreement with the
theoretical analysis in Section \ref{sec:generalize}: SAT-BR2 represents a
tighter stability bound and cannot tolerate as small a relaxation factor.  As
$\alpha$ is reduced in Figure~\ref{fig:relaxation}, the eigenvalue corresponding
to BR2 becomes positive before that of SIPG for the same SBP operator.
Additionally, for SBP-$\Gamma$, a higher order approximation allows a lower
relaxation factor, while for SBP-$\Omega$ the opposite trend is observed.

\begin{figure} 
    \begin{subfigure}{0.51\textwidth}
        \centering
        \includegraphics[width=1.0\linewidth]{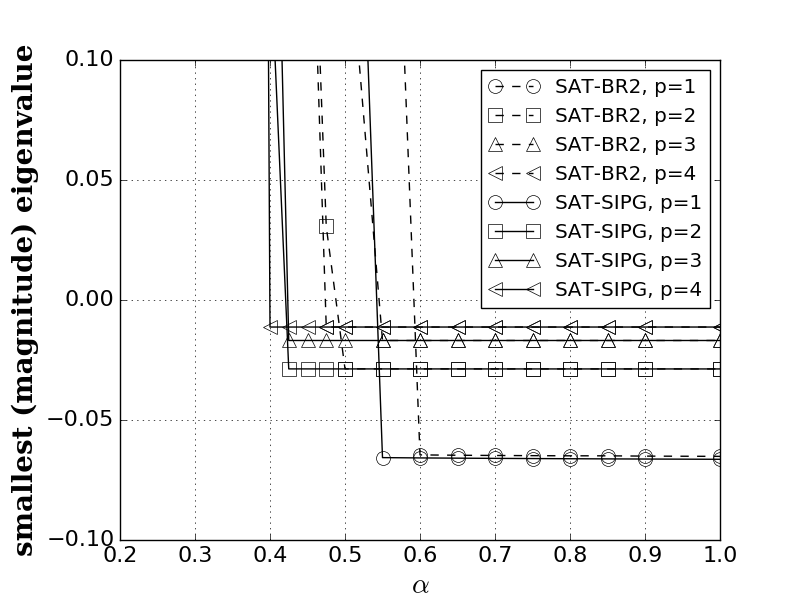}
        \subcaption{Smallest eigenvalue using SBP-$\Gamma$}
        \label{fig:relaxation_cn_gamma}
    \end{subfigure}
    \begin{subfigure}{0.51\textwidth}
        \centering
        \includegraphics[width=1.0\linewidth]{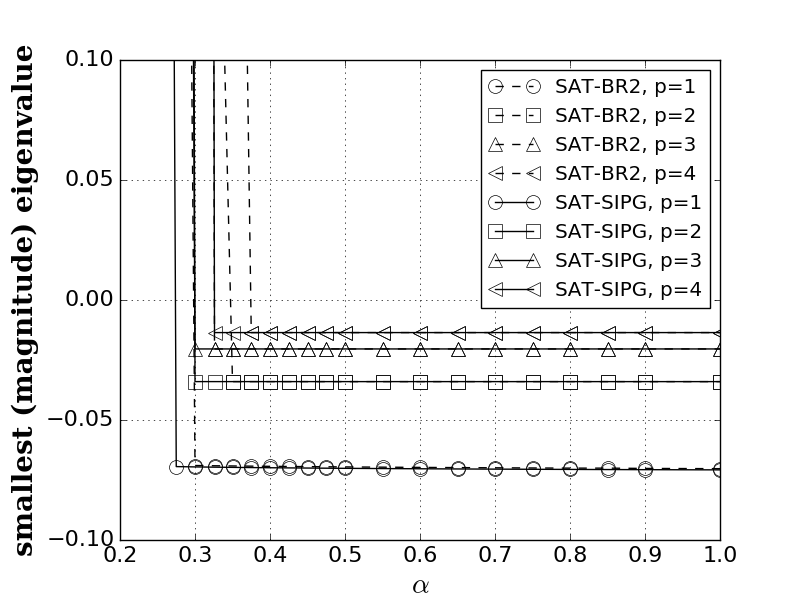}
        \subcaption{Smallest eigenvalue using SBP-$\Omega$}
        \label{fig:relaxation_cn_omega}
    \end{subfigure} 
    \caption{Relaxation effect on SATs}
    \label{fig:relaxation}   
\end{figure}

\subsection{Unsteady energy stability}

In this section, to complement the preceding investigation, we solve an unsteady
problem with homogeneous Dirichlet boundary conditions and no source term using
a SBP-$\Gamma$ discretization with two different relaxation factors: a stable
value of $\alpha = 1$ and an unstable value of $\alpha=0.3$ (based on the
results in Figure~\ref{fig:relaxation_cn_gamma}).  As mentioned in Section
\ref{sec:energy_analysis}, the PDE solution ``energy'' should be monotonically
decreasing as time evolves.  The solution of the SBP-SAT discretization will
also have a decreasing energy, provided the stability conditions of Theorem
\ref{thm:stability-condition} are satisfied.


For this study, the time derivative is discretized using the second-order
backward differentiation formula (BDF2) with a time step $\Delta
t=10^{-3}$. Once again, the $16\times 16$ perturbed mesh shown in Figure
\ref{fig:perturbed_mesh} is utilized. The same random solution is prescribed as the initial condition for the same order of approximations.

Figure \ref{fig:unsteady_energy} shows the energy evolution using different
order SBP-$\Gamma$ operators with SAT-BR2 penalties. The results using
SBP-$\Omega$ and/or SAT-SIPG are qualitatively the same and are not shown
here. As can be seen, all solutions based on $\alpha=0.3$ diverge after a short
period of time; note the logarithmic time scale.  In contrast, the energy for
solutions of the unscaled SAT-BR2 discretizations (\ie, $\alpha=1$) is
monotonically decreasing, as expected.  Additionally, we see that the $p=2$,
$p=3$, and $p=4$ discretizations produce indistinguishable results, while $p=1$
produces a slightly different energy history.  This suggests that operators
above degree one are necessary to resolve the energy history of the homogeneous
model problem on this particular mesh.

\begin{figure} 
    \centering
        \centering
        \includegraphics[width=1.0\linewidth]{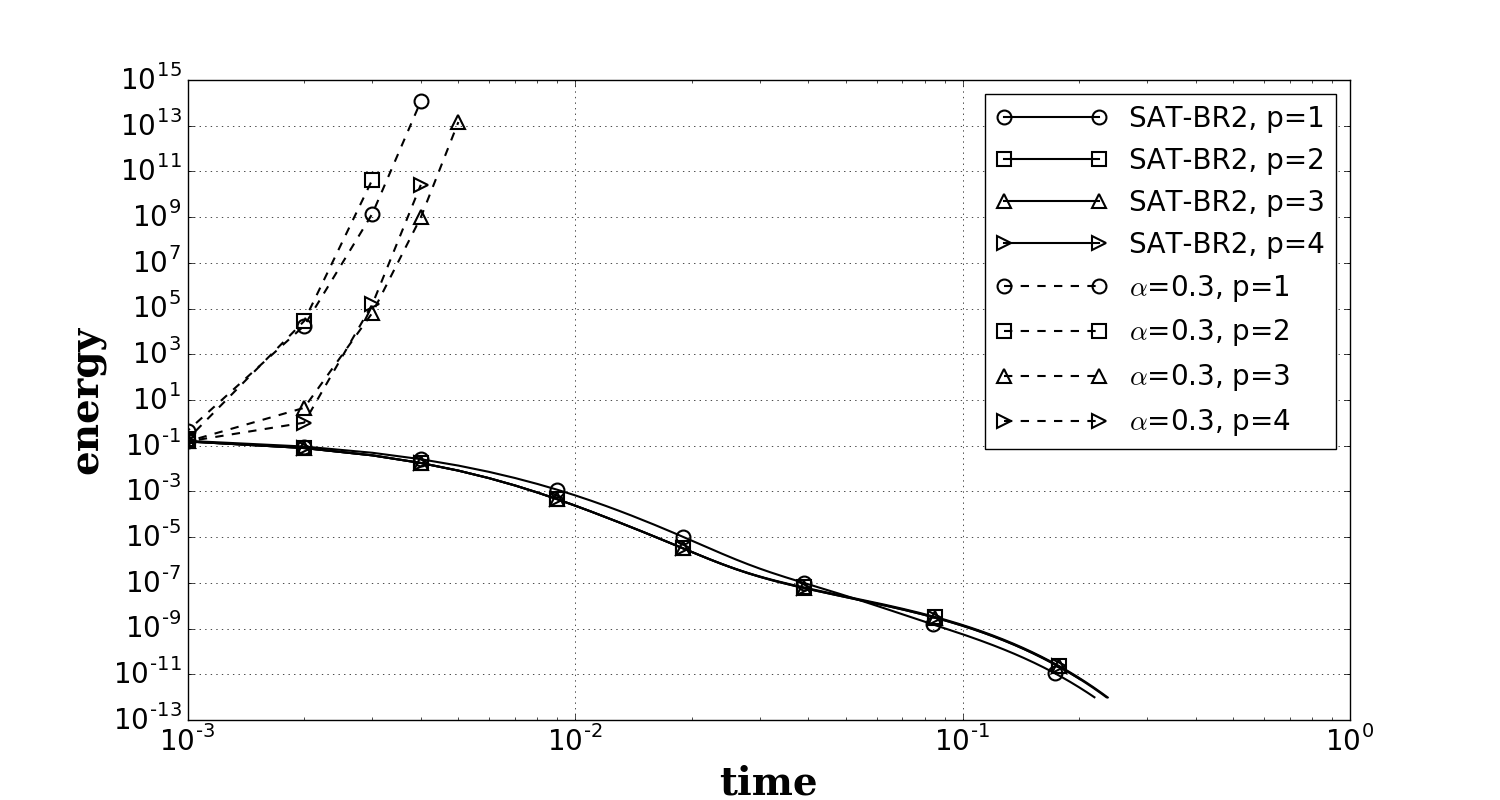}
        \caption{Energy history of homogeneous problem using BDF2}
        \label{fig:unsteady_energy}
\end{figure}

\section{Summary and Conclusions}\label{sec:conclude}

We generalized the SAT methodology to accommodate multi-dimensional SBP
discretizations of second-order PDEs, including SBP operators whose volume nodes
do not coincide with a boundary cubature.  We considered a general form of SAT
that uses dense penalty coefficient matrices on each face of the SBP elements.
Starting with this general framework, we carried out analyses of adjoint
consistency and energy stability, and, based on these analyses, we determined
conditions on the coefficient matrices that guarantee a conservative,
energy-stable, primal-consistent, and adjoint-consistent discretization.

In contrast with previous finite-element analyses of interior penalties, the SAT
conditions given here apply to general (tensor) diffusion coefficients and
arbitrary elements.  Furthermore, the conditions are entirely algebraic.  Using
the properties of SBP operators, our analysis accounts for inexact integration
explicitly from the beginning such that numerical instability caused by aliasing
errors is avoided.

Two popular interior penalty methods used in the FE community, BR2 and SIPG,
were generalized to multi-dimensional SBP-SAT discretizations.  We demonstrated
that the SIPG penalty can be obtained from BR2 using straightforward matrix
analysis; to the best of our knowledge, this algebraic connection has not been
previously reported.

Several numerical test cases were carried out to verify the analysis and compare
the performance of SAT-BR2 and SAT-SIPG when applied in conjunction with two
families of SBP operators: the so-called SBP-$\Gamma$ and SBP-$\Omega$
operators.  Mesh refinement studies confirmed that the discretizations achieve
design order and that they produce superconvergent functionals.  Comparisons
between different discretizations were carried out on an extremely skewed mesh.
The results suggest that the SBP-$\Gamma$ operators produce better conditioned
systems than the SBP-$\Omega$ operators.  Our stability bound was shown to be
relatively tight in the sense that a scaling factor applied to one of the SATs
could not be reduced below one order of magnitude without causing instability.

\bibliographystyle{aiaa}
\bibliography{references}

\end{document}